\documentclass{amsart}

\usepackage{amsthm,amssymb,latexsym,amsmath,amsfonts,amscd,amstext}
\usepackage[all]{xy}
\usepackage[english]{babel}
\usepackage[utf8x]{inputenc} \usepackage[T1]{fontenc}
\usepackage{color}
\usepackage{graphicx}	
\usepackage{mathrsfs}
\usepackage[all]{xy}

\usepackage{latexsym}
\usepackage{epsfig}
\usepackage[mathscr]{eucal}

\usepackage{tikz}
\usepackage{tikz-cd}
\usetikzlibrary{arrows,automata}
\usetikzlibrary{fit}
\usepackage{wrapfig}

\newcommand{\p}[1]{{\mathbb{P}^{#1}}}
\newcommand{\calf}{{\mathcal F}}
\newcommand{\calm}{{\mathcal M}}
\newcommand{\caln}{{\mathcal N}}
\newcommand{\calo}{{\mathcal O}}
\newcommand{\C}{\mathbb{C}}
\newcommand{\FF}{\mathbf{F}}
\newcommand{\MM}{\mathbf{F}}
\newcommand{\QQ}{\mathbf{F}}

\DeclareMathOperator{\im}{{im}}
\DeclareMathOperator{\Hom}{Hom}

\DeclareMathOperator{\Kom}{Kom}

\newtheorem{thm}{Theorem}

\newtheorem{cor}[thm]{Corollary}
\newtheorem{lem}[thm]{Lemma}
\newtheorem{pps}[thm]{Proposition}
\newtheorem{dfn}[thm]{Definition}
\newtheorem{obs}[thm]{Remark}

\title{Moduli spaces of framed flags of sheaves on the projective plane}

\author{R. A. von Flach}
\address{IMECC - UNICAMP \\
Departamento de Matem\'atica \\ Rua S\'ergio  \\
13083-970 Campinas-SP, Brazil}
\email{vonflachrodrigo@gmail.com}

\author{Marcos Jardim}
\address{IMECC - UNICAMP \\
Departamento de Matem\'atica \\ Rua S\'ergio  \\
13083-970 Campinas-SP, Brazil}
\email{jardim@ime.unicamp.br}
	
\thanks{This work is part of the first author's PhD thesis. RvF acknowledges the support by CAPES and CNPQ. MJ is partially supported by the CNPq grant number 303332/2014-0 and the FAPESP grant number 2014/14743-8.}

\begin{document}
 
\begin{abstract}
We study the moduli space of framed flags of sheaves on the projective plane via an adaptation of the ADHM construction of framed sheaves. In particular, we prove that, for certain values of the topological invariants, the moduli space of framed flags of sheaves is an irreducible, nonsingular variety carrying a holomorphic pre-symplectic form. 

\end{abstract}

\maketitle

\tableofcontents

\section{Introduction}

Moduli spaces of flags of sheaves played an important role first in the work of Grojnowski \cite{G} and Nakajima \cite{nakajima}, and later in the higher rank generalization due to Baranovsky \cite{B}, in which these authors construct an action of the Heisenberg algebra in the cohomology of moduli spaces of sheaves on surfaces. More recently, flags of sheaves also appeared in the work of Bruzzo et al. \cite{A01-01} in the context of a supersymmetric quantum mechanical model in string theory, and in the work of Chuang et al. \cite{CDDT} where a string theoretic derivation for the conjecture of Hausel, Letellier and Rodriguez-Villegas on the cohomology of character varieties with marked points is given.

Our motivation to provide, in the present paper, an ADHM construction of the moduli space of framed flags of sheaves on the projective plane. More precisely, we consider triples $(E,F,\varphi)$ consisting of a torsion free sheaf $F$ on $\p2$, a framing $\varphi$ of $F$ at a line $\ell_\infty$ and a subsheaf $E$ of $F$ such that the quotient $F/E$ is supported away from the framing line $\ell_\infty$. Letting $r:={\rm rk}(E)={\rm rk}(F)$, $n:=c_2(E)$, and $l:=h^0(F/E)$, we denote by $\calf(r,n,l)$ the moduli space of such triples with these invariants fixed. Building upon the techniques used in \cite[Section 3]{A01-01}, we prove that $\calf(r,n,1)$ is an irreducible, nonsingular quasi-projective variety of dimension $2rn+r+1$. Note also that $\calf(1,n,l)$ coincide with the nested Hilbert scheme ${\rm Hilb}^{n,n+l}(\C^2)$ of points in $\C^2$, so our main result can be regarded as a generalization to higher rank of known facts about ${\rm Hilb}^{n,n+l}(\C^2)$.

From a geometric point of view, we show that $\calf(r,n,1)$ admits the structure of a holomorphic pre-symplectic manifold, that is, $\calf(r,n,1)$ is a K\"ahler manifold equipped with a natural closed holomorphic 2-form $\Omega$. In the simplest possible case, namely $r=n=1$, we show that $\Omega$ is generically non-degenerate.

Our approach is to relate framed flags of sheaves with representations of the \emph{enhanced ADHM quiver}
\begin{equation}
\label{enhanced-ADHM-quiver}
\begin{tikzpicture}[->,>=stealth',shorten >=1pt,auto,node distance=2.5cm,
semithick]
\tikzstyle{every state}=[fill=white,draw=none,text=black]

\node[state]    (A)                {$e_1$};
\node[state]    (B) [left of=A]    {$e_2$};
\node[state]    (C) [right of=A]   {$e_{\infty}$};

\path (A) edge [loop above]          node {$\alpha$}     (A)
edge [loop below]          node {$\beta$}      (A)
edge [out=20,in=160]       node {$\eta$}       (C)
edge [out=-160,in=-20]     node {$\gamma$}     (B)
(B) edge [loop above]          node {$\alpha'$}    (B)
edge [loop below]          node {$\beta'$}     (B)
edge [out=20,in=160]       node {$\phi$}       (A)
(C) edge [out=-160,in=-20]     node {$\xi$}        (A);
\end{tikzpicture}
\end{equation}
with the relations
\begin{equation}
\begin{array}{rccc}
\label{eq:enhADHMrel02}
\alpha'\beta'-\beta'\alpha'& \alpha\beta - \beta\alpha + \xi\eta, & \alpha\phi-\phi\alpha', & \beta\phi-\phi\beta',\\
\eta\phi, & \begin{array}{lr}
\gamma\xi, & \phi\gamma, \end{array} & \gamma\alpha-\alpha'\gamma, & \gamma\beta - \beta'\gamma.
\end{array}
\end{equation}
Indeed, we show that the moduli space of stable representations of the enhanced ADHM quiver with dimension vector $(r,c,c')$ is isomorphic to $\calf(r,c-c',c')$. All properties of the former space are obtained by analyzing the moduli space of stable quiver representations.

The paper is outlined as follows. Section \ref{sec.flags} is dedicated to the functorial construction of the moduli space of framed flags of sheaves.  The enhanced ADHM quiver and its stable representations are introduced in Section \ref{e-adhm-q}, and the moduli space stable representations is constructed in Section \ref{section:moduli-space}. We prove that $\calf(r,n,1)$ is an irreducible, nonsingular quasi-projective variety of dimension $2rn+r+1$ in Sections \ref{smoothness} and \ref{adhm const} by first establishng the corresponding facts for the moduli space of stable representation of the enhanced ADHM quiver, and then showing that this is isomorphic to the moduli space of framed flags of sheaves. Finally, Section \ref{geometry} is dedicated to the study of the geometric structure of $\calf(r,n,1)$.


\section{Framed flags of sheaves on $\p2$} \label{sec.flags}

Fix a line $\ell_\infty\subset\p2$; recall that a \emph{framing} of a coherent sheaf $F$ on $\p2$ at the line $\ell_\infty$ is the choice of an isomorphism $\varphi:F|_{\ell\infty}\to{\mathcal O}_{\ell\infty}^{\oplus r}$, where $r$ is the rank of $F$.
A \emph{framed flag of sheaves} on $\p2$ is a triple $(E,F,\varphi)$ consisting of a torsion free sheaf $F$ on $\p2$, a framing $\varphi$ of $F$ at the line $\ell_\infty$, and a subsheaf $E$ of $F$ such that the quotient $F/E$ is supported away from the framing line $\ell_\infty$. Note that the existence of a framing forces $c_1(F)=0$, while the last condition implies that $c_1(E)=0$, and that $F/E$ must be a 0-dimensional sheaf. Thus the triple $(E,F,\varphi)$ has three numerical invariants: $r:={\rm rk}(E)={\rm rk}(F)$, $n:=c_2(F)$ and $l:=h^0(F/E)$; note that $c_2(E)=n+l$. 

We consider the moduli functor
$$ \FF_{(r,n,l)} ~:~ {\rm Sch}_{\C}^{\rm op} \longrightarrow {\rm Sets} $$
$$ \FF_{(r,n,l)}(S) =
\left\{ \textrm{isomorphism classes of quadruples} (F_S,\varphi_S,Q_S,g_S) \right\} $$
where
\begin{itemize}
\item[(i)] $F_S$ is a coherent sheaf on $\p2\times S$, flat over $S$ such that
$F_s:=F_S|_{\p2\times\{s\}}$ is a torsion free sheaf for every closed point $s\in S$ with
${\rm rk}(F_s)=r$, $c_1(F_s)=0$, and $c_2(F_s)=n$;

\item[(ii)] $\varphi_S:E_S|_{\ell_\infty\times S} \to {\mathcal O}_{\ell_\infty\times S}^{\oplus r}$
is an isomorphism of ${\mathcal O}_{\ell_\infty\times S}$-modules;

\item[(iii)] $Q_S$ is a coherent sheaf on $\p2\times S$, flat over $S$ supported away from
$\p2\times\ell_{\infty}$, and such that $h^0(Q_s)=l$ for every closed point $s\in S$, where
$Q_s:=Q_S|_{\p2\times\{s\}}$;

\item[(iv)] $g_S:F_S\to Q_S$ is a surjective morphism of ${\mathcal O}_{\p2\times S}$-modules.

\end{itemize}

Two quadruples $(F_S,\varphi_S,Q_S,g_S)$ and $(F_S',\varphi_S',Q_S',g_S')$ are isomorphic if there are isomorphisms of ${\mathcal O}_{\p2\times S}$-modules $\Theta_S:F_S\to F'_S$ and $\Gamma_S:Q_S\to Q'_S$ such that the following two diagrams commute
$$ \xymatrix{
F_S|_{\ell_\infty\times S} \ar[r]^{\varphi_S} \ar[d]^{\Theta_S|_{\ell_\infty\times S}} &
{\mathcal O}_{\ell_\infty\times S}^{\oplus r} 
& ~~ & ~~ & F_S \ar[r]^{g_S}\ar[d]^{\Theta_S} & Q_S \ar[d]^{\Gamma_S} \\
F_S'|_{\ell_\infty\times S} \ar[ru]^{\varphi'_S} & & ~~ & ~~ & F'_S \ar[r]^{g'_S} & Q_S
} $$

\begin{pps}
The functor $\FF_{(r,n,l)}$ is representable.
\end{pps}

The variety that represents functor $\FF_{(r,n,l)}$ is called the \emph{moduli space of framed flags of sheaves on $\p2$}; it will be denoted by $\calf(r,n,l)$.

Before going into the proof, let us remind a few facts about moduli spaces of framed sheaves, following \cite{BM,HL}. Consider the moduli functor
$$ \MM_{(r,n)} ~:~ {\rm Sch}_{\C}^{\rm op} \longrightarrow {\rm Sets} $$
$$ \FF_{(r,n)}(S) =
\left\{ \textrm{isomorphism classes of } (F_S,\varphi_S) \right\} $$
where
\begin{itemize}
\item[(i)] $F_S$ is a coherent sheaf on $\p2\times S$, flat over $S$ such that
$F_s:=F_S|_{\p2\times\{s\}}$ is a torsion free sheaf for every closed point $s\in S$ with
${\rm rk}(F_s)=r$, $c_1(F_s)=0$, and $c_2(F_s)=n$;

\item[(ii)] $\varphi_S:E_S|_{\ell_\infty\times S} \to {\mathcal O}_{\ell_\infty\times S}^{\oplus r}$
is an isomorphism of ${\mathcal O}_{\ell_\infty\times S}$-modules;
\end{itemize}
This functor is represented by a quasi-projective variety $\calm(r,n)$, which is the moduli space of framed torsion free sheaves on $\p2$. Let $(U,\epsilon)$ denote the universal framed sheaf on $\p2\times\calm(r,n)$, with
$\epsilon: U|_{\ell_\infty\times\calm(r,n)} \stackrel{\simeq}{\longrightarrow} \calo_{\ell_\infty\times\calm(r,n)}^{\oplus r}$ being an isomorphism of $\calo_{\ell_\infty\times\calm(r,n)}$-modules.

\begin{proof}
Consider the quot functor
$$ \QQ_{(U,l)} ~:~ {\rm Sch}_{\calm(r,n)}^{\rm op} \longrightarrow {\rm Sets} $$
$$ \QQ_{(U,l)}(S) =
\left\{ \textrm{isomorphism classes of } (Q_S,g_S) \right\} $$ 
where
\begin{itemize}
\item[(i)] $Q_S$ is a coherent sheaf on $\p2\times S$, flat over $S$, supported away from 
$\ell_\infty\times S$ and such that $h^0(Q_s)=l$ for every closed point $s\in S$, where
$Q_s:=Q_S|_{\p2\times\{s\}}$;

\item[(ii)] $q_S:({\mathbf 1}_{\p2}\times\pi)^*U \to Q_S$ is a surjective morphism of $\calo_{\p2\times S}$-modules, where ${\mathbf 1}_{\p2}$ is the identity map, and $\pi:S\to\calm(r,n)$;
\end{itemize}
By Gothendieck's general theory, $\QQ_{(U,l)}$ is representable. 

Next, we argue that the functors $\FF_{(r,n,l)}$ and $\QQ_{(U,l)}$ are isomorphic, implying that the former is also representable.

Indeed, there is a natural transformation $\FF_{(r,n,l)}\to \QQ_{(U,l)}$ defined by 
defined by $(F_S,\varphi_S,Q_S,g_S)\mapsto(Q_S,g_S)$, with inverse given by setting
$F_S = \ker(g|_S)$ and noting that, as consequence of the condition on the support of $Q_S$, the framing $\epsilon$ of the universal sheaf $U$ on $\ell_\infty\times\calm(r,n)$ induces a framing
$\varphi_S$ on $F_S$.
\end{proof}

It follows from the proof that the moduli space of framed flags of sheaves actually coincides with the quot scheme ${\rm Quot}^l(U)$ relative to $\calm(r,n)$. This observation provides a forgetful morphism of schemes $\mathfrak{p}:\calf(r,n,l)\to\calm(r,n)$, given by $(E,F,\varphi)\mapsto(F,\varphi)$. 

\begin{lem}\label{lem2}
The morphism $\mathfrak{p}$ is surjective, and its fibre over $(F,\varphi)$ is the open subset of the quot scheme ${\rm Quot}^l(F)$ consisting of those sheaves supported away from the framing line $\ell_\infty\subset\p2$. 
\end{lem}

\begin{proof}
Take a framed sheaf $(F,\varphi)\in\calm(r,n)$, and let $\Sigma:={\rm Sing}(F)$ be its singular locus. Choose $l$ distinct points $x_1,\dots,x_l\in\p2\setminus\Sigma$, and let
$Q:=\bigoplus_{i=1}^l\calo_{x_i/\p2}$. Then there exist epimorphisms $\alpha:F\twoheadrightarrow Q$; choose one such epimorphism and let $E:=\ker\alpha$. The triple $(E,F,\varphi)$ defines a point in the fibre $\mathfrak{p}^{-1}(F,\varphi)$. The second claim is an immediate consequence of the isomorphism between ${\rm Quot}^l(U)$ and $\calf(r,n,l)$ as schemes over $\calm(r,n)$.
\end{proof}

There is, of course, another forgetful morphism $\calf(r,n,l)\to\calm(r,n+l)$ given by $(E,F,\varphi)\mapsto(E,\varphi)$. This, however, is not dominant for $l\ge 2$, since the sheaf $E$ is never locally free. Remark also that $\calf(r,n,l)$ lies within the product $\calm(r,n)\times\calm(r,n+l)$ as an incidence variety.


\section{Quiver representations and stability conditions} \label{e-adhm-q}

In order to further understand the geometrical properties of the moduli spaces $\calf(r,n,l)$, we shift our attention to representations of the ADHM quiver.

A \textit{representation of the enhanced ADHM quiver of type $(r,c,c')$} in the category of complex vector spaces is given by the set $X=(W,V,V',A,B,I,J,A',B',F,G)$, where $W$, $V$, $V'$ are vector spaces of complex dimension $r$, $c$ and $c'$, respectively, and $A,$ $B\in End(V)$, $I\in Hom(W,V)$, $J\in Hom(V,W)$, $A'$, $B'\in End(V')$, $F\in Hom(V',V)$ and $G\in Hom(V,V')$ that satisfy the following equations called \textit{enhanced ADHM equations}

\begin{equation}
\label{eq:enhADHMequations}
\begin{array}{rcccc}
[A,B]+IJ=0, & [A',B']=0, & AF-FA'=0, & BF-FB'=0, & JF=0,\\
GI=0, & FG=0,& GA-A'G=0, & GB-B'G=0. &
\end{array}
\end{equation}
A representation $X = (W,V,V',A,B,I,J,A',B',F,G)$ can be illustrated as the diagram below
\begin{equation*}
\begin{tikzpicture}[->,>=stealth',shorten >=1pt,auto,node distance=2.5cm,semithick]
\tikzstyle{every state}=[fill=white,draw=none,text=black]

\node[state]    (A)                {$V$};
\node[state]    (B) [left of=A]    {$V'$};
\node[state]    (C) [right of=A]   {$W$.};

\path (A) edge [loop above]          node {$A$}        (A)
edge [loop below]          node {$B$}        (A)
edge [out=20,in=160]       node {$J$}        (C)
edge [out=-160,in=-20]     node {$G$}        (B)
(B) edge [loop above]          node {$A'$}       (B)
edge [loop below]          node {$B'$}       (B)
edge [out=20,in=160]       node {$F$}        (A)
(C) edge [out=-160,in=-20]     node {$I$}        (A);
\end{tikzpicture}
\end{equation*}

Let $\varphi: W \longrightarrow \mathbb{C}^r$ be an isomorphism. Then, if $X$ is a representation of the enhanced ADHM, $(X,\varphi)$ is called a \textit{framed representation} of the enhanced ADHM quiver \label{framed-representation}. Two framed representations $(X,\varphi)$ and $(\widetilde{X},\widetilde{\varphi})$ are said to be isomorphic if there exists an isomorphism
\begin{equation*}
(\xi_1,\xi_2,\xi_{\infty}): X \longrightarrow \widetilde{X},
\end{equation*}
such that $\widetilde{\varphi}\xi_{\infty}=\varphi$. \\

In order to construct the moduli space, first we need to introduce a stability condition. In this section we will define two stability conditions and we will prove that these conditions are equivalent in a suitable chamber. The first one, called $\Theta$-stability condition, was inspired in the stability condition presented by King in 1994 in \cite{A01-01-20}. The $\Theta$-stability condition is a good one because there exists techniques using Geometric Invariant Theory to construct the moduli space of $\Theta$-stable framed representations of quivers, as we prove in the Section \ref{section:moduli-space}.\\
\indent While the second stability condition is more resembling with the stability condition usually defined for representations of the ADHM quiver. This resemblance plays an important role to prove that the moduli space stable framed representation of the enhanced ADHM quiver of numerical type $(r,c,c')$, where $c'\geq 2$, is not smooth. 

\begin{dfn}
	\index{Representation of the enhanced ADHM quiver!$\Theta-$stable}\index{Representation of the enhanced ADHM quiver!$\Theta-$semistable}
	Let $\Theta=(\theta,\theta',\theta_{\infty})\in \mathbb{Q}^3$ a triple satisfying the relation
	\begin{equation}
	\label{eq:stability-parameter}
	c\theta + c'\theta' + r\theta_{\infty} = 0.
	\end{equation}
	A representation $X$ of numerical type $(r,c,c')$ is called \textit{$\Theta-$stable} if $X$ satisfies the following conditions:
	\begin{itemize}
		\item[$(i)$] Any subrepresentation $0\neq\widetilde{X}\subset X$ of numerical type $(0,\widetilde{c},\widetilde{c'})$ satisfies
		\begin{equation}
		\label{eq:thetastab01}
		\theta\widetilde{c}+\theta'\widetilde{c'}< 0;
		\end{equation}
		\item[$(ii)$] Any subrepresentation $0\neq\widetilde{X}\subset X$ of numerical type $(\widetilde{r},\widetilde{c},\widetilde{c'})$ satisfies
		\begin{equation}
		\label{eq:thetastab02}
		\theta_{\infty}\widetilde{r}+\theta\widetilde{c}+\theta'\widetilde{c'}< 0.
		\end{equation}
	\end{itemize}
	A representation $X$ of numerical type $(r,c,c')$ is called \textit{$\Theta-$semistable} if $X$ satisfies
	\begin{itemize}
		\item[$(iii)$] Any subrepresentation $0\neq\widetilde{X}\subset X$ of numerical type $(0,\widetilde{c},\widetilde{c'})$ satisfies
		\begin{equation}
		\label{eq:thetastab03}
		\theta\widetilde{c}+\theta'\widetilde{c'}\leq 0;
		\end{equation}
		\item[$(iv)$] Any subrepresentation $0\neq\widetilde{X}\subset X$ of numerical type $(\widetilde{r},\widetilde{c},\widetilde{c'})$ satisfies
		\begin{equation}
		\label{eq:thetastab04}
		\theta_{\infty}\widetilde{r}+\theta\widetilde{c}+\theta'\widetilde{c'}\leq 0.
		\end{equation}
	\end{itemize}\end{dfn}
	
	Note that this stability condition is slightly different from the notion defined by King, since here only subrepresentations of numerical type $(0,c,c')$ and $(r,c,c')$ was considered. However, this notion defined in \cite{A01-01} is enough to construct the moduli space of $\Theta$-stable framed representations of the enhanced ADHM quiver, as we will see in Section \ref{section:moduli-space}. Let $(r,c,c')$ be a fixed dimension vector. Then the space of stability parameters $\Theta=(\theta,\theta',\theta_{\infty})\in \mathbb{Q}^3$ satisfying \eqref{eq:stability-parameter} can be identified  with the $(\theta,\theta')-$plane in $\mathbb{Q}^2$, after solving for $\theta_{\infty}$. If the set of the representations $X$ with numerical type $(r,c,c')$ which is strictly $\Theta-$semistable is nonempty, the parameter $\Theta$ is called  \textit{critical of type} $(r,c,c')$. Otherwise, if this set is empty, the parameter $\Theta$ is called \textit{generic}. The following lemma establishes the existence of generic stability parameters for any given dimension vector $(r,c,c')$. Moreover, this lemma is analogous to \cite[Lemma 3.1]{A01-01} and a proof is written only for completeness.
	\begin{lem}
		\label{lem:thetastab-equiv-stab}
		Fix a triple $(r,c,c')\in(\mathbb{Z}_{>0})^3$. Suppose $\theta'>0$ and $\theta+c'\theta'<0$. Let $X=(W,V,V',A,B,I,J,A',B',F,G)$ be a representation of numerical type $(r,c,c')$. Then the following are equivalent:\begin{itemize}
			\item[$($i$)$] $X$ is $\Theta-$stable;
			\item[$($ii$)$] $X$ is $\Theta-$semistable;
			\item[$($iii$)$] $X$ satisfies the following conditions:
			\begin{enumerate}
				\item[$(S.1)$] $F\in Hom(V',V)$ is injective;
				\item[$(S.2)$] The ADHM data $\mathcal{A}=(W,V,A,B,I,J)$ is stable, i.e., there is no proper subspace $0\subset S\subsetneq V$ preserved by $A$, $B$ and containing the image of $I$.
			\end{enumerate}
		\end{itemize}
	\end{lem}
	\begin{proof}If $X$ is $\Theta-$stable, then $X$ is clearly $\Theta-$semistable.\\
		\indent Suppose that $X$ is $\Theta-$semistable and $F$ is not injective. Then
		\begin{align}
		A'(\ker(F))\subset \ker(F) \nonumber\\
		B'(\ker(F))\subset \ker(F) \nonumber 
		\end{align}
		In fact, let $v\in\ker(F)$. Then it follows from the enhanced ADHM equations that
		\begin{align}
		0 & = (BF-FB')v \nonumber\\
		& \Rightarrow F(B'v) = 0\mbox{, for all } v\in\ker(F) \nonumber \\
		& \Rightarrow B'(v)\in\ker(F) \mbox{, for all }v\in\ker(F) \nonumber\\
		& \Rightarrow B'(\ker(F))\subset \ker(F) \nonumber
		\end{align}
		Analogously the same can be proved for the endomorphism $A'$.\\
		\indent Then, $\widetilde{X} = (0,0,\ker(F),0,0,0,0,A'|_{\ker(F)},B'|_{\ker(F)},F|_{\ker(F)}, 0)$ is a subrepresentation of $X$ with numerical type $(\widetilde{r},\widetilde{c},\widetilde{c'})$ in which $\widetilde{r}=\widetilde{c}=0$ and $\widetilde{c'}=\dim(\ker(F))$. However,
		\begin{align*}
		\widetilde{c}\theta+\widetilde{c'}\theta' = \theta'\cdot\dim(\ker(F)) > 0
		\end{align*}
		and this contradicts the inequation \eqref{eq:thetastab03}.\\
		\indent Now suppose that $X$ is $\Theta-$semistable and the condition $(S.2)$ is false. Then, there is a proper subspace $0\subset S \subsetneq V$ such that
		\begin{equation*}
		A(S), B(S), Im(I)\subseteq S.
		\end{equation*}
		Therefore, $\widetilde{X}=(W,S,V',A|_S,B|_S,I,J|_S,0,0,0,0)$ is a subrepresentation with numerical type $(r,\dim(S),c')$. However, it follows from the equation \eqref{eq:stability-parameter} and from conditions $\theta'>0$, $\theta + c'\theta'<0$ that 
		$$
		\dim(S)\theta + c'\theta'+\theta_{\infty}r> 0.$$
		Indeed, 
		\begin{align}
		\theta'>0,~\theta + c'\theta'<0 \Rightarrow \theta<0, \nonumber
		\end{align}
		moreover,
		\begin{align}
		\dim(S)\theta + c'\theta'+\theta_{\infty}r &= (\dim(S)-c)\theta >0
		\end{align} 
		and this contradicts the inequality \eqref{eq:thetastab04}. Thus, if $X$ is $\Theta-$semistable, then $X$ satisfies the conditions $(S.1)$ and $(S.2)$.\\
		\indent Now suppose that $X$ satisfies the conditions $(S.1)$ and $(S.2)$, thus $X$ is $\Theta-$stable. Indeed, let $\widetilde{X}=(\widetilde{W},\widetilde{V},\widetilde{V'},\widetilde{A},\widetilde{B},\widetilde{I},\widetilde{J},\widetilde{A'},\widetilde{B'},\widetilde{F},\widetilde{G})$ be a subrepresentation of $X$ with numerical type $(\widetilde{r},\widetilde{c},\widetilde{c'})$. There are two cases to study: $\widetilde{r}=r$ and $\widetilde{r}=0$.\\
\indent		First suppose $\widetilde{r}=r$. Since $\widetilde{W}$ is a subspace of $W$, by definition, and $\widetilde{r}=r$, $\widetilde{W}=W$. It follows from condition $(S.2)$ that $I\neq 0$. Indeed, otherwise, $0\subset V$ would satisfy $A(0), B(0), Im(0) = 0$ and then $\mathcal{A}=(W,V,A,B,I,J)$ is not stable. Since $\widetilde{X}$ is a subrepresentation of $X$, the diagram below commutes:
		\begin{equation*}
		\begin{tikzcd}
		W \arrow{r}{I}                            & V \\
		W \arrow{u}{1_{W}} \arrow{r}{\widetilde{I}} & \widetilde{V} \arrow[hook]{u}{\iota}
		\end{tikzcd}
		\end{equation*}
		i.e.,
		\begin{align}
		\label{eq:subrep-theta-stab}
		\iota\circ \widetilde{I}= I\circ 1_{W} .
		\end{align}
		Thus, if $\widetilde{c}=0$,  $I\equiv 0$, which is a contradiction. Therefore $\widetilde{c}>0$. \\
		\indent If $\widetilde{c}<c$, then $0\subset \widetilde{V} \subsetneq V$ is a proper subspace such that 
		$$
		A(\widetilde{V}), B(\widetilde{V}), Im(I) \subset \widetilde{V}.
		$$
		Indeed, since $\widetilde{X}$ is a subrepresentation of $X$, the following diagram commutes:
		\begin{equation*}
		\begin{tikzcd}
		V \arrow{r}{A}                            & V \\
		\widetilde{V} \arrow[hook]{u}{\iota} \arrow{r}{\widetilde{A}} & \widetilde{V} \arrow[hook]{u}{\iota}
		\end{tikzcd}
		\end{equation*}
		thus, 
		\begin{equation*}
		\iota\circ\widetilde{A}=A\circ \iota \Rightarrow A|_{\widetilde{V}} \subset \widetilde{V}.
		\end{equation*}
		Analogously, $B$ preserves $\widetilde{V}$. Moreover, it follows from \eqref{eq:subrep-theta-stab} that $Im(I)\subset \widetilde{V}$, and this contradicts the condition $(S.2)$. Therefore, $\widetilde{c}=c$. Since $\widetilde{X}$ is a proper subrepresentation, $\widetilde{c'}<c'$ and then
		\begin{align}
		\theta\widetilde{c} + \theta'\widetilde{c'} + \theta_{\infty}\widetilde{r} &= \theta c + \theta'\widetilde{c'} + \theta_{\infty}r -(\theta c + \theta'c'+\theta_{\infty}r) \nonumber \\
		&= \theta'(\widetilde{c'}-c') <0 \nonumber
		\end{align}
		
		Now suppose $\widetilde{r}=0$. If $\widetilde{c}=0$, since $F$ is injective, $\widetilde{V'}\subset\ker(F)=0$. Thus $\widetilde{V'}=0$. However, only nontrivial subrepresentations are being considered. Then $\widetilde{c}>0$. If $\widetilde{c}<c$, again $0\subset\widetilde{V}\subsetneq V$ contradicts the condition $(S.2)$ and $\widetilde{c}=c$. Since $\widetilde{X}$ is a proper subrepresentation, $\widetilde{c'}<c$. Thus
		\begin{align*}
		\widetilde{c}\theta + \widetilde{c'}\theta' &\leq \theta+c'\theta' < 0
		\end{align*}
		Therefore, $X$ is $\Theta-$stable.
	\end{proof}
	The following corollary is trivial.
	\begin{cor}
		\label{corolario-lema222}
		Let $X=(A,B,I,J,A',B',F,G)$ be a stable representation of the enhanced ADHM quiver. Then $G\equiv 0$.
	\end{cor}
	\begin{proof}
		Note that $F$ is injective, see condition $(S.1)$, and $FG=0$, see enhanced ADHM equations \eqref{eq:enhADHMequations}. Then $G=0$.
	\end{proof}
	\indent Due to the Lemma \ref{lem:thetastab-equiv-stab}, from now on a representation $X$ of the enhanced ADHM quiver will be called \textit{stable} \index{Representation of the enhanced ADHM quiver!stable} if $X$ satisfies $(S.1)$ and $(S.2)$. Due to Corollary \ref{corolario-lema222}, framed stable representations are quite simpler and easier to manipulate.\\
	\indent In this work almost always the representations of the enhanced ADHM quiver considered are stable. Thus, sometimes it will be considered the following quiver as the \textit{enhanced ADHM quiver}\index{Enhanced ADHM quiver}
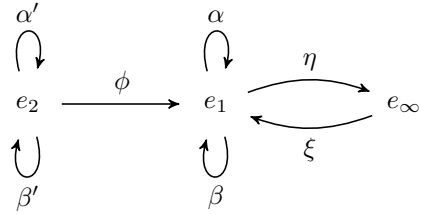
	\begin{equation*}\label{eq:quiverADHMaument-est}
	\begin{tikzpicture}[->,>=stealth',shorten >=1pt,auto,node distance=2.5cm,
	semithick]
	\tikzstyle{every state}=[fill=white,draw=none,text=black]
	
	\node[state]    (A)                {$e_1$};
	\node[state]    (B) [left of=A]    {$e_2$};
	\node[state]    (C) [right of=A]   {$e_{\infty}$};
	
	\path (A) edge [loop above]          node {$\alpha$}     (A)
	edge [loop below]          node {$\beta$}      (A)
	edge [out=20,in=160]       node {$\eta$}       (C)
	(B) edge [loop above]          node {$\alpha'$}    (B)
	edge [loop below]          node {$\beta'$}     (B)
	edge [out=0,in=180]       node {$\phi$}       (A)
	(C) edge [out=-160,in=-20]     node {$\xi$}        (A);
	\end{tikzpicture}
	\end{equation*}
	with ideal generated by relations
	\begin{equation*}
	\begin{array}{ccccc}
	\label{eq:quiverADHMenh}
	\alpha\beta - \beta\alpha + \xi\eta, & \alpha\phi -\phi\alpha', & \beta\phi-\phi\beta', &\eta\phi, & \alpha'\beta'-\beta'\alpha'.
	\end{array}
	\end{equation*}
	Then a representation of the quiver above is given by $X=(A,B,I,J,A',B',F)$ such that $A,$ $B\in End(V)$, $I\in Hom(W,V)$, $J\in Hom(V,W)$, $A'$, $B'\in End(V')$ and $F\in Hom(V',V)$, see the diagram below,
	
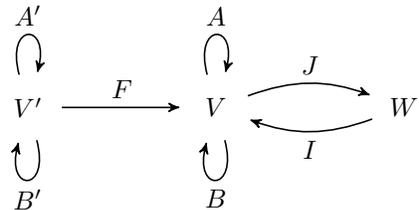
\begin{equation*}
	\begin{tikzpicture}[->,>=stealth',shorten >=1pt,auto,node distance=2.5cm,
	semithick]
	\tikzstyle{every state}=[fill=white,draw=none,text=black]
	
	\node[state]    (A)                {$V$};
	\node[state]    (B) [left of=A]    {$V'$};
	\node[state]    (C) [right of=A]   {$W$};
	
	\path (A) edge [loop above]     node {$A$}     (A)
	edge [loop below]     node {$B$}     (A)
	edge [out=20,in=160]      node {$J$}     (C)
	(B) edge [loop above]     node {$A'$}    (B)
	edge [loop below]     node {$B'$}    (B)
	edge [out=0, in=180]      node {$F$}     (A)
	(C) edge [out=-160,in=-20]      node {$I$}     (A);
	\end{tikzpicture}
	\end{equation*}
	satisfying the equations
	\begin{equation}
	\begin{array}{ccc}	\label{eq:repquiverADHMenh}
	~~ [A,B]+IJ = 0, & JF=0, & ~~ \\ 
	~~ [A',B']=0, & AF-FA'=0, & BF-FB'=0,
	\end{array}
	\end{equation}
	which will be also called \textit{enhanced ADHM equations} in this work.


\section{Construction of the moduli space} \label{section:moduli-space}

In this section we present the construction of the moduli spaces of framed $\Theta-$semistable representation of the enhanced ADHM quiver. In order to do this, we will use Geometric Invariant Theory techniques by analogy of \cite{A01-01-20} and \cite[Section 3.2]{A01-01}. This construction is presented in details just for completeness.\\ 
	\indent Let $V$, $V'$, $W$ be complex vector spaces such that $\dim V = c$, $\dim V' = c'$ and $\dim W = r$. We define the \index{Enhanced ADHM!data} space of \textit{enhanced ADHM data}, denoted by $\mathbb{X}$ or $\mathbb{X}(r,c,c')$, as the following complex vector space
	\begin{align}
	\mathbb{X} =& End(V)^{\oplus^2}\oplus Hom(W,V) \oplus Hom (V,W)\oplus\nonumber\\
	            & \oplus End(V')^{\oplus^2}\oplus Hom(V',V)\oplus Hom(V,V')\nonumber.
	\end{align}
	A vector $X\in \mathbb{X}$, $X= (A,B,I,J,A',B',F,G)$, is called \index{Enhanced ADHM!datum} \textit{enhanced ADHM datum}. Let $$\mathcal{G}=GL(V)\times GL(V').$$ Consider the map
	\begin{equation}
	\begin{array}{cccl}
	\label{eq:acaolivre}
	\Psi: &\mathcal{G}\times \mathbb{X} & \longrightarrow & \mathbb{X}\\
	& (h,h',X)                     & \longmapsto     & (hAh^{-1}, hBh^{-1}, hI, Jh^{-1}, h'A'h'^{-1},\\
	&&& h'B'h'^{-1}, hFh'^{-1},h'Gh^{-1}).
	\end{array}
	\end{equation} 
	
	It is easy to check that the map $\Psi$ defines a $\mathcal{G}-$action on $\mathbb{X}$.
	\begin{pps}
		\label{prop:acaolivre}
		The $\mathcal{G}-$action (\ref{eq:acaolivre}) is free on stable points of the enhanced ADHM data $\mathbb{X}$.
	\end{pps}
	\begin{proof}
		Suppose that there exists $(h,h')\in \mathcal{G}$ such that
		\begin{align}
		(h,h')\cdot X = X, \text{ for all } X\in \mathbb{X}\nonumber.
		\end{align}
		Then, the following equations are satisfied
		\begin{align}
		\label{eq:acaolivre01}hAh^{-1}  & = A, \quad hA = Ah;\\
		\label{eq:acaolivre02}hBh^{-1}  & = B, \quad hB = Bh;\\
		\label{eq:acaolivre03}hI        & = I, \quad (h-1_{V}) I = 0;\\
		hA'h^{-1} & = A', \quad hA' = A'h;\nonumber\\
		hB'h^{-1} & = B',\quad hB' = B'h;\nonumber\\
		\label{eq:acaolivre06}hFh'^{-1}  & = F 
		\end{align}
		Let $S:=\ker(h-1_{V})$. It follows from the equation (\ref{eq:acaolivre03}) that $Im(I)\subset S$.Furthermore, let $v\in S$. Then, 
		\begin{align}
		\label{eq:acaolivre07}
		(h-1_{V})v & = 0 \quad \Rightarrow \quad hv = v.
		\end{align}
		Therefore, it follows from the equations (\ref{eq:acaolivre01}), (\ref{eq:acaolivre02}) and (\ref{eq:acaolivre07}) that $A(S), B(S)\subset S$.
		It follows from the stability condition of the ADHM datum $(A,B,I,J)$ that $S = V$. Then, 
		\begin{equation}
		\label{eq:acaolivre08}
		h=1_V.
		\end{equation} We claim that $h'=1_{V^{\prime}}$. Indeed, it follows from the equations (\ref{eq:acaolivre06}) and (\ref{eq:acaolivre08}) that $F(1_{V^{\prime}}-h'^{-1})=0$. Since $F$ is injective, it follows that $(1_{V^{\prime}}-h'^{-1}) = 0$. Therefore, $h'=1_{V'}$ and $(h,h') = (1_V,1_{V^{\prime}})$, which concludes the proof.
	\end{proof}
	The \textit{stabilizer} of a given point $X\in\mathbb{X}$ is denoted by $\mathcal{G}_X\subset\mathcal{G}$. It is easy to check the following.
	\begin{lem}
		\label{lem:preservedbyG}
		Let $\mathbb{X}_0=\mathbb{X}_0(r,c,c')\subset\mathbb{X}(r,c,c')$ be the subscheme defined by the equations \eqref{eq:enhADHMequations}. Then $\mathbb{X}_0$ is preserved by the $\mathcal{G}$-action \eqref{eq:acaolivre}.
	\end{lem} 
	\begin{proof}
		If $X=(A,B,I,J,A',B',F,G)\in \mathbb{X}_0$ and $(h,h')\in\mathcal{G}$, then 
		$$
		(h,h')\cdot X = (hAh^{-1}, hBh^{-1}, hI, Jh^{-1}, h'A'h'^{-1}, h'B'h'^{-1}, hFh'^{-1},h'Gh^{-1}).
		$$
		Furthermore, it follows from the equations \eqref{eq:enhADHMequations} that $(h,h')\cdot X$ also satisfies the equations \eqref{eq:enhADHMequations}. In other words, $(h,h')\cdot X\in \mathbb{X}_0$ and $\mathbb{X}_0$ is preserved by the $\mathcal{G}$-action \eqref{eq:acaolivre} in consequence.
	\end{proof}
	\begin{obs}
		Each representation $X = (W,V,V',A,B,I,J,A',B',F,G)$ corresponds to a datum vector $X\in\mathbb{X}_0$. Moreover, two framed representations are isomorphic if and only if the corresponding points in $\mathbb{X}_{0}$ are in the same orbit.  
	\end{obs}
	
	Now we will recall of a few results about Geometric Invariant Theory for representations of quivers. More details can be found in \cite{A01-01-20}. First, the notion of $\chi$-(semi)stability for a given character $\chi:\mathcal{G}\longrightarrow \mathbb{C}^{\ast}$ will be defined. Then it will be proved that the notion of $\Theta$-(semi)stability is equivalent to the notion of $\chi$-(semi)stability for a specific character that will be defined below. 
	\begin{dfn}
		 \index{$\chi$-stable}\index{$\chi$-semistable}
		Let $\mathcal{G}$ be a reductive algebraic group acting on a vector space $\mathbb{X}$. Given an algebraic character
		$$
		\chi:\mathcal{G}\longrightarrow \mathbb{C}^{\ast},
		$$
		a point $X_0\in\mathbb{X}$ is called:
		\begin{itemize}
			\item[(i)] \textit{$\chi$-semistable} \index{$\chi$-semistable}, if there exists a polynomial function $p(X)$ on $\mathbb{X}(r,c,c')$ satisfying:
			\begin{align}
			\label{eq:poly-func-prop}
			p((h,h')\cdot X_0) = \chi(h,h')^lp(X_0),
			\end{align}
			for some $l\in\mathbb{Z}_{\geq 1}$, such that $p(X_0)\neq 0$;
			\item[(ii)] \textit{$\chi$-stable}\index{$\chi$-stable}, if there exists a polynomial function $p(X)$ on $\mathbb{X}(r,c,c')$ satisfying \eqref{eq:poly-func-prop} for some $l\in\mathbb{Z}_{\geq 1}$, such that $p(X_0)\neq 0$ and such that
			$$
			\dim(\mathcal{G}\cdot X_0)=\dim(\mathcal{G}/\Delta),
			$$
			where $\Delta\subset\mathcal{G}$ is the subgroup which acts trivially on $\mathbb{X}$, and the action of $\mathcal{G}$ on $\{x\in\mathbb{X}: p(x)\neq 0\}$ is closed.
		\end{itemize}
	\end{dfn}
	The next Lemma gives to us an equivalent definition of $\chi$-(semi)stability.
	\begin{lem}
		Let $\mathcal{G}$ act on the direct product $\mathbb{X}_0(r,c,c')\times\mathbb{C}$ by 
		\begin{align*}
		(h,h')\times(X,z)\longmapsto ((h,h')\cdot X,\chi(h,h')^{-1}z).
		\end{align*}
		A point $X\in\mathbb{X}$ is
		\begin{itemize}
			\item[(i)] $\chi$-semistable if and only if the closure of the orbit $\mathcal{G}\cdot(X,z)$ is disjoint from the zero section $\mathbb{X}(r,c,c')\times\{0\}$, for all $z\neq 0$;
			\item[(ii)] $\chi$-stable if and only if the orbit is closed in the complement of the zero section, and the stabilizer $G_{(X,z)}$ is a finite index subgroup of $\Delta$.
		\end{itemize}
	\end{lem}
	The proof of this Lemma can be found in \cite[Lemma 2.2]{A01-01-20}. One can form the quasi-projective scheme:
	\begin{equation}
	\label{eq:nthetasemistable}
	\mathcal{N}_{\chi}^{ss}(r,c,c')=\mathbb{X}_0(r,c,c')//_{\chi}\mathcal{G}:=\mbox{Proj}(\oplus_{n\geq 0}A(\mathbb{X}_0(r,c,c'))^{\mathcal{G},\chi^{n}}),
	\end{equation}
where
$$ A(\mathbb{X}_0(r,c,c'))^{\mathcal{G},\chi^{n}}:= \{f\in A(\mathbb{X}_0(r,c,c')) ~|~ f((h,h')\cdot X)=\chi(h,h')^nf(X), $$
$$ ~~~~~~~~~~~~~~~~~~~~~~~~ \mbox{ for all }(h,h')\in\mathcal{G}\}. $$
\begin{obs}
		It is well known that $\mathcal{N}_{\chi}^{ss}(r,c,c')$ is projective over Spec$(\mathbb{X}_0(r,c,c')^{\mathcal{G}})$, and it is quasi-projective over $\mathbb{C}$. Geometric Invariant Theory says that $\mathcal{N}^{ss}_{\chi}(r,c,c')$ is the space of $\chi$-semistable orbits; moreover it contains an open subscheme \linebreak $\mathcal{N}^s_{\theta}(r,c,c')\subseteq\mathcal{N}^{ss}_{\theta}(r,c,c')$ consisting of $\chi$-stable orbits.
	\end{obs}
	The following proposition holds by analogy with \cite[Proposition 3.1, Theorem 4.1]{A01-01-20} and the proof is analogous to \cite[Proposition 3.1]{A01-01}. The proof is given in details just for completeness.
	\begin{pps}
		\label{prop:chi-theta-stab-iff-theta-stab}
		Suppose that $\Theta=(\theta,\theta')\in\mathbb{Z}^2$ and let $\chi_{\Theta}:\mathcal{G}\longrightarrow\mathbb{C}^{\ast}$ be the character
		\begin{equation*}
		\chi_{\Theta}(h,h')=\det(h)^{-\theta}\det(h')^{-\theta'}.
		\end{equation*}
		Let $X=(W,V,V',A,B,I,J,A',B',F,G)$ be a representation of the enhanced ADHM quiver and $X$ the corresponding point in $\mathbb{X}_0$. Thus, $X$ is $\Theta$-(semi)stable if and only if $X$ is $\chi_{\Theta}$-(semi)stable.
	\end{pps}
	\begin{proof}
Suppose that $X$ is $\chi_{\Theta}$-semistable. Let $\theta_{\infty}\in \mathbb{Z}$ such that it satisfies \eqref{eq:stability-parameter}. Suppose that there exists a nontrivial proper subrepresentation
$$ \widetilde{X}=(\widetilde{W},\widetilde{V},\widetilde{V'},\widetilde{A},\widetilde{B},\widetilde{I},\widetilde{J},\widetilde{A'},\widetilde{B'},\widetilde{F},\widetilde{G}) $$
of numerical type $(r,c,c')$ of the representation $X$ such that $\widetilde{r}=\dim(\widetilde{W})\in\{0,\mbox{ }r\}$ satisfies
		\begin{align*}
		\widetilde{c}\theta+\widetilde{c'}\theta'+\widetilde{r}\theta_{\infty} > 0.
		\end{align*}
		First take $\widetilde{r}=0$. Then $\widetilde{W}=\{0\}$. Since $\widetilde{X}$ is a subrepresentation of $X$, $\widetilde{V}$ and $\widetilde{V'}$ are subspaces of $V$ and $V'$, respectively, and it follows that
		\begin{align*}
		F(\widetilde{V'})\subseteq\widetilde{V}, \quad G(\widetilde{V})\subseteq\widetilde{V'},\quad A(\widetilde{V}),\mbox{ }B(\widetilde{V})\subseteq\widetilde{V},\\ A'(\widetilde{V'}),\mbox{ }B'(\widetilde{V'})\subseteq\widetilde{V'},\quad J(\widetilde{V})=0. 
		\end{align*}
		Thus, there exist direct sum decompositions 
		\begin{equation}
		\label{eq:dir-sum-decomp}
		\left\{\begin{array}{c}
		V\cong\widetilde{V}\oplus\widehat{V}\\
		V'\cong\widetilde{V'}\oplus\widehat{V'}
		\end{array}\right.\end{equation}
		such that the linear maps $A,$ $B,$ $A'$, $B'$, $F$, $G$ have block decomposition of the form 
		\begin{equation}
		\label{eq:abfab-block-form-dec}
		\left[\begin{array}{cc}
		\ast & \ast\\
		0    & \ast
		\end{array}\right]
		\end{equation}
		while the linear maps $I$ and $J$ have block decomposition of the form
		\begin{equation*}
		I =\left[\begin{array}{c}
		\ast\\
		\ast
		\end{array}\right],\quad J= \left[\begin{array}{cc}
		0    & \ast
		\end{array}\right].
		\end{equation*}
		Consider a one-parameter subgroup of $\mathcal{G}$ of the form
		\begin{equation}
		\label{eq:1-par-subg-r0}
		h(t) = \left[\begin{array}{cc}
		t1_{\widetilde{V}} & 0\\
		0    & 1_{\widehat{V}}
		\end{array}\right],\quad h'(t) = \left[\begin{array}{cc}
		t1_{\widetilde{V}} & 0\\
		0    & 1_{\widehat{V'}}
		\end{array}\right].
		\end{equation}
		It follows that the linear maps
		\begin{align*}
		(A(t),B(t),I(t),J(t),A'(t),B'(t),F(t),G(t)) = (h(t),h'(t))\cdot X
		\end{align*}
		have block decomposition of the form
		\begin{equation}
		\label{eq:abfabt-block-dec}
		\left[\begin{array}{cc}
		\ast & t\ast\\
		0    & \ast
		\end{array}\right]
		\end{equation}
		and
		\begin{equation}
		\label{eq:ij-block-decomp-r0}
		I(t) = \left[\begin{array}{c}
		t\ast\\
		\ast
		\end{array}\right], \quad J(t) = \left[\begin{array}{cc}
		0    & \ast
		\end{array}\right].
		\end{equation}
		However,
		\begin{equation*}
		\begin{array}{lcl}
		\chi_{\Theta}(h(t),h'(t))\cdot z & = & \left(\det\left[\begin{array}{cc}
		t1_{\widetilde{V}} & 0\\
		0    & 1_{\widehat{V}}
		\end{array}\right]^{-\theta}\det\left[\begin{array}{cc}
		t1_{\widetilde{V}} & 0\\
		0    & 1_{\widehat{V'}}
		\end{array}\right]^{-\theta'}\right)^{-1}\cdot z \\
		& = & (t^{-\theta\widetilde{c}-\theta'\widetilde{c'}})^{-1}\cdot z\\
		& = & t^{\theta\widetilde{c}+\theta'\widetilde{c'}}\cdot z
		\end{array}
		\end{equation*}
		with $\theta\widetilde{c}+\theta'\widetilde{c'}>0$. Therefore, 
		\begin{align*}
		\lim_{t\rightarrow 0} (h(t),h'(t)\cdot(X,z))\in \mathbb{X}\times\{0\},
		\end{align*} 
		which contradicts the $\chi_{\Theta}$-semistability condition.\\
		\indent Now suppose $\widetilde{r}=r$. Thus, analogously to the case $r=0$, one can obtain 
		\begin{align*}
		F(\widetilde{V'})\subseteq\widetilde{V}, G(\widetilde{V})\subseteq\widetilde{V'},\quad A(\widetilde{V}),\mbox{ }B(\widetilde{V})\subseteq\widetilde{V}, \quad A'(\widetilde{V'}),\mbox{ }B'(\widetilde{V'}),\quad I(\widetilde{W})\subseteq\widetilde{V'}. 
		\end{align*}
		Therefore, there exist direct sum decompositions like in \eqref{eq:dir-sum-decomp} such that the maps $A,$ $B,$ $A',$ $B'$, $F$ and $G$ have block form decomposition of the form \eqref{eq:abfab-block-form-dec}, while $I$, $J$ have block form decompositions of the form
		\begin{equation}
		\label{eq:ij-block-decomp-rr}
		I =\left[\begin{array}{c}
		\ast\\
		0
		\end{array}\right],\quad J= \left[\begin{array}{cc}
		\ast    & \ast
		\end{array}\right].
		\end{equation}
		Now consider a one-parameter subgroup of $\mathcal{G}$ of the form
		\begin{equation}
		\label{eq:one-par-subg-rr}
		h(t) = \left[\begin{array}{cc}
		1_{\widetilde{V}} & 0\\
		0    & t^{-1}1_{\widehat{V}}
		\end{array}\right],\quad h'(t) = \left[\begin{array}{cc}
		t1_{\widetilde{V}} & 0\\
		0    & t^{-1}1_{\widehat{V'}}
		\end{array}\right].
		\end{equation}
		It follows that the linear maps
		\begin{align*}
		(A(t),B(t),I(t),J(t),A'(t),B'(t),F(t),G(t)) = (h(t),h'(t))\cdot X
		\end{align*}
		have block decomposition of the form \eqref{eq:abfabt-block-dec} and
		\begin{equation}
		\label{eq:ijt-block-dec-rr}
		I(t) = \left[\begin{array}{c}
		\ast\\
		0
		\end{array}\right], \quad J(t) = \left[\begin{array}{cc}
		\ast    & t\ast
		\end{array}\right].
		\end{equation}
		However,
		\begin{equation*}
		\begin{array}{lcl}
		\chi_{\Theta}(h(t),h'(t))\cdot z & = & \left(\det\left[\begin{array}{cc}
		1_{\widetilde{V}} & 0\\
		0    & t^{-1}1_{\widehat{V}}
		\end{array}\right]^{-\theta}\det\left[\begin{array}{cc}
		t1_{\widetilde{V}} & 0\\
		0    & t^{-1}1_{\widehat{V'}}
		\end{array}\right]^{-\theta'}\right)^{-1}\cdot z \\
		& = & (t^{(\widetilde{c}-c)\theta+(\widetilde{c'}-c')\theta'})\cdot z
		\end{array}
		\end{equation*}
		in which $((\widetilde{c}-c)\theta+(\widetilde{c'}-c')\theta')>0$. Indeed
		\begin{equation*}
		\begin{array}{lcl}
		(\widetilde{c}-c)\theta+(\widetilde{c'}-c')\theta' & = & (\widetilde{c}\theta+\widetilde{c'}\theta')\theta_{\infty}- (c\theta+c'\theta')\\
		& > & -\widetilde{r}\theta_{\infty} + r\theta_{\infty}\\
		& = & 0
		\end{array}
		\end{equation*}
		
		Therefore, 
		\begin{align*}
		\lim_{t\rightarrow 0}(h(t),h'(t)\cdot(X,z))&= \lim_{t\rightarrow 0}t^{((\widetilde{c}-c)\theta+(\widetilde{c'}-c')\theta')}\cdot Z\nonumber\\
		&= 0\in \mathbb{X}\times\{0\},
		\end{align*} 
		which contradicts the $\chi_{\Theta}$-semistability condition since the closure of the orbit intersects $\mathbb{X}(r,c,c')\times\{0\}$ for some $z\neq 0$. Then, in both cases, $\chi_{\Theta}$-semistability implies $\Theta$-semistability.\\
		\indent Now suppose that $X$ is $\chi_{\Theta}$-stable but it is not $\Theta$-stable. In particular if $X$ is $\chi_{\Theta}$-stable, $X$ is $\chi_{\Theta}$-semistable and thus $\Theta$-semistable in consequence. Therefore, there exists a proper subrepresentation $\widetilde{X}$ of $X$ with numerical type $(\widetilde{r},\widetilde{c},\widetilde{c'})$ such that $\widetilde{r}\in\{0,r\}$ and
		$$
		\widetilde{c}\theta+\widetilde{c'}\theta'+\widetilde{r}\theta_{\infty}=0.
		$$
		
		There are two cases to consider, $\widetilde{r}=0$ and $\widetilde{r}=r$. In both cases, it will be proved that $X$ has a nontrivial stabilizer, which contradicts the $\chi_{\Theta}$-stability condition.\\
		\indent First, consider $\widetilde{r}=0$. As above, $A$, $B$, $F$, $G$, $A'$ and $B'$ have block decomposition of the form \eqref{eq:abfab-block-form-dec} the direct sum decomposition of $V$ and $V'$ like \eqref{eq:dir-sum-decomp}. Consider a one-parameter subgroup, $(h(t),h'(t))$, of $\mathcal{G}$ of the form \eqref{eq:1-par-subg-r0}. The linear maps $(A(t),B(t),I(t),J(t),A'(t),B'(t),F(t),G(t)) = (h(t),h'(t))\cdot (X,z)$ have block form decomposition of the form \eqref{eq:abfabt-block-dec} and \eqref{eq:ij-block-decomp-r0}. Therefore, the limit of $(h(t),h'(t))\cdot (X,z)$ as $t\rightarrow 0$ has block decomposition of the form
		\begin{align*}
		\left[\begin{array}{cc}
		\ast & 0\\
		0    & \ast
		\end{array}\right], \mbox{for $A(t)$, $B(t)$, $A'(t)$, $B'(t)$, $F(t)$ and $G(t)$},\\ I(t)=\left[\begin{array}{c}
		0\\
		\ast
		\end{array}\right],\quad J(t)=\left[\begin{array}{cc}
		0 & \ast
		\end{array}\right]
		\end{align*}
		On the other hand, since $\mathcal{G}\cdot(X,z)$ is closed for $z\neq 0$, the linear maps $A$, $B$, $A'$, $B'$, $F$, $G$ must have block decomposition of the form
		\begin{equation*}
		\left[\begin{array}{cc}
		\ast & 0\\
		0    & \ast
		\end{array}\right]
		\end{equation*}
		while $I$, $J$ have block decomposition of the form 
		\begin{equation*}
		I=\left[\begin{array}{c}
		0\\
		\ast
		\end{array}\right],\quad J=\left[\begin{array}{cc}
		0 & \ast
		\end{array}\right].
		\end{equation*}
		
		Thus, the subgroup $(h(t),h'(t))$ stabilizes $(X,z)$ which contradicts the $\chi_{\Theta}$-stability condition.\\
		\indent Now consider $\widetilde{r}=r$. One can repeat the step above obtaining the block decomposition in \eqref{eq:abfab-block-form-dec} for the linear maps $A$, $B$, $A'$, $B'$, $F$, and $G$, while $I$ and $J$ have the block decomposition in \eqref{eq:ij-block-decomp-rr}. Thus, let $(h(t),h'(t))$ be a one-parameter subgroup of $\mathcal{G}$ of the form \eqref{eq:one-par-subg-rr}. Then, the linear maps
		$$
		(A(t),B(t),I(t),J(t),A'(t),B'(t),F(t),G(t)) = (h(t),h'(t))\cdot X
		$$
		
		are such that $A(t)$, $B(t)$, $A'(t)$, $B'(t)$, $F(t)$ and $G(t)$ have block decomposition of the form \eqref{eq:abfabt-block-dec} while
		$I(t)$ and $J(t)$ have block decomposition of the form \eqref{eq:ijt-block-dec-rr}. Therefore, the limit of $(h(t),h'(t))\cdot(X,z)$ as $t\rightarrow 0$ have block decomposition of the form
		\begin{align*}
		\left[\begin{array}{cc}
		\ast & 0\\
		0    & \ast
		\end{array}\right], \mbox{for $A(t)$, $B(t)$, $A'(t)$, $B'(t)$, $F(t)$ and $G(t)$},\\ I(t)=\left[\begin{array}{c}
		\ast\\
		0
		\end{array}\right],\quad J(t)=\left[\begin{array}{cc}
		\ast & 0
		\end{array}\right].
		\end{align*}
		Again, this implies that $(X,z)$ have a nontrivial stabilizer leading to a contradiction.\\
		\indent The other side of this proof is analogous.
\end{proof}
	
Therefore, it follows from Lemma \ref{lem:thetastab-equiv-stab} and Proposition \ref{prop:chi-theta-stab-iff-theta-stab} that there exists a chamber in $\mathbb{Q}^{2}$ given by $\theta'> 0$ and $\theta + c'\theta'<0$ such that all the stability conditions defined until now are the same. Thus, given a representation of the enhanced ADHM quiver $X$ with numerical type $(r,c,c')$ and $\Theta=(\theta,\theta',\theta_{\infty})$ satisfying $\theta'> 0$ and $\theta + c'\theta'<0$ from now on $X$ will be called \textit{stable}\index{Representation of the enhanced ADHM quiver!stable} if it satisfies one of the conditions below:
	\begin{itemize}
		\item[$(i)$] X satisfies the conditions $(S.1)$ and $(S.2)$ of the Lemma \ref{lem:thetastab-equiv-stab};
		\item[$(ii)$] X is $\Theta$-stable;
		\item[$(iii)$] X is $\Theta$-semistable;
		\item[$(iv)$] X is $\chi_{\Theta}$-stable;
		\item[$(v)$] X is $\chi_{\Theta}$-semistable.
	\end{itemize}
	Then, in a suitable chamber, the moduli space of framed stable representations of numerical type $(r,c,c')$ of the enhanced ADHM quiver denoted, $\mathcal{N}(r,c,c')= \mathcal{N}_{\chi}^{ss}(r,c,c')$ is given by the equation \eqref{eq:nthetasemistable}.\\
	\indent \label{x''01}For further reference, let $X=(W,V,V',A,B,I,J,A',B',F)$ be a framed stable representation of numerical type $(r,c,c')$ of the enhanced ADHM quiver. One can consider the stable representation of the ADHM quiver
$$ X''=(W,V'',A'',B'',I'',J'') $$
of numerical type $(r,c-c')$, where $V'':=V/Im(F)$ and the maps $A''$, $B''\in End(V'')$, $I\in Hom(W,V'')$ and $J\in Hom(V'',W)$ are inherited by the quotient $V/Im(F)$. Moreover, $X''$ is indeed stable and satisfies the ADHM equation $[A'',B'']+I''J''=0$. See the proof below.\\
	\indent Consider a basis in $V$ such that 
	\begin{equation*}
	F = \left[\begin{array}{c}
	1_{V'}\\
	0
	\end{array}\right].
	\end{equation*}
	Let 
	\begin{equation*}
	\begin{array}{cccc}
	\pi':& V'\oplus V'' & \longrightarrow & V'\\
	&  (v',v'')    & \longmapsto     & v' \end{array},\quad \begin{array}{cccc}
	\pi'':& V'\oplus V'' & \longrightarrow & V''\\
	&  (v',v'')    & \longmapsto     & v''
	\end{array}.\end{equation*} 
	Then $V$ can be decomposed as $V=V'\oplus V/Im(F)=V'\oplus V''$ and $A''$, $B''$, $I''$, $J''$ are given by
	\begin{equation*}
	A''=A|_{V''},\quad B''=B|_{V''},\quad I''=\pi''\circ I,\quad J''=J|_{V''}.
	\end{equation*}
	Therefore, 
	\begin{align}
	[A'',B''] + I''J'' &= [A|_{V''},B|_{V''}] + (\pi''\circ I)\circ(J|_{V''}) \nonumber \\
	&= [A,B]|_{V''}+ (IJ)|_{V''} \nonumber \\
	&= ([A,B]+IJ)|_{V''} \nonumber \\
	&= 0
	\end{align}
	and $X''=(W,V'',A'',B'',I'',J'')$ is stable. Indeed, suppose that there exists $0\subset S'' \subsetneq V''$ a subspace of $V''$ such that 
	\begin{equation*}
	A''(S''),\quad B''(S''),\quad I''(W) \subset S''.
	\end{equation*}
	Then $0\subset V'\oplus S'' \subsetneq V$ is a subspace such that
	\begin{equation*}
	A(V'\oplus S''),\quad B(V'\oplus S''),\quad I(W)\subset V'\oplus S''.
	\end{equation*}
	In fact, fix $(v',s'')\in V'\oplus S''$. Thus
	\begin{align}
	A(v',s'') &= (A|_{V'}(v'),A|_{V''}(s'')) \nonumber \\
	&\in V'\oplus S'' \nonumber,
	\end{align}
	which means that $A(v',s'')\in V'\oplus S''$ for all $(v',s'')\in V'\oplus S''$, i.e., $A(V'\oplus S'')\subset V'\oplus S''$.
	Analogously one can obtain that $B(V'\oplus S'')\subset V'\oplus S''$. Moreover, 
	\begin{align}
	I(W) &= I(W)\cap V' \oplus I(W)\cap V'' \nonumber \\
	&= \pi'\circ I(W)\oplus \pi''\circ I(W) \nonumber \\
	&\subset V'\oplus S''
	\end{align}
	which contradicts the condition $(S.2)$ of Lemma \ref{lem:thetastab-equiv-stab}.\\
	\indent Therefore\label{x''02}, if $X=(W,V,V',A,B,I,J,A',B',F)$ is a framed stable representation of the enhanced ADHM quiver with numerical type $(r,c,c')$, then \linebreak $X''=(W,X'',A'',B'',I'',J'')$ is a stable representation of the ADHM quiver of numerical type $(r,c-c')$.\\

The following Lemma is analogous to \cite[Lemma 3.2]{A01-01}; we include its proof here for the sake of completeness.

\begin{lem} \label{lem:q-sobrejetiva}
Let $\mathcal{M}(r,c-c')$ be the moduli space of the stable representations of the ADHM quiver of numerical type $(r,c-c')$. There exists a surjective morphism
\begin{equation*}	\begin{array}{cccc}
\mathfrak{q}: & \mathcal{N}(r,c,c')    & \longrightarrow & \mathcal{M}(r,c-c')\\
& [(W,V,V',A,B,I,J,A',B',F)]  & \longmapsto     & [(W,V'',A'',B'',I'',J'')]
\end{array} \end{equation*}
where $[(W,V,V',A,B,I,J,A',B',F)]$ and $[(W,V'',A'',B'',I'',J'')]$ denote the isomorphism class of the framed stable representation $(W,V,V',A,B,I,J,A',B',F)$ of the enhanced ADHM quiver and the isomorphism class of the stable representation $(W,V'',A'',B'',I'',J'')$ of the ADHM quiver constructed above, respectively.
\end{lem}

\begin{proof}
The construction above shows the existence of the morphism $\mathfrak{q}$. It is enough to prove that this morphism is surjective. So, fix an ADHM data $(A'',B'',I'',J'')$ of numerical type $(r,c-c')$ and the morphisms $A'$, $B'\in End(V')$. Set $V=V'\oplus V''$ and 
		\begin{equation*}
		F = \left[\begin{array}{c}
		1_{V'}\\
		0
		\end{array}\right].
		\end{equation*}
		Now let $A,$ $B\in End(V)$, $I\in Hom(W,V)$ and $J\in Hom(V,W)$ be of the following form
		\begin{equation*}
		A=\left[\begin{array}{cc}
		A' & \widetilde{A} \\
		0  & A''
		\end{array}\right],\quad B=\left[\begin{array}{cc}
		B' & \widetilde{B}\\
		0  & B''
		\end{array}\right],\quad I =\left[\begin{array}{c}
		\widetilde{I}\\
		I''
		\end{array}\right],\quad J =\left[\begin{array}{cc}
		0 & J''
		\end{array}\right],\quad 
		\end{equation*}
		according to the decomposition $V=V'\oplus V''$. This means that
		\begin{equation*}
		\widetilde{A},\mbox{ } \widetilde{B}\in Hom(V'',V')\quad \mbox{and}\quad \widetilde{I}\in Hom(W,V').
		\end{equation*}
		It is easy to check that
		\begin{equation}
		\label{eq:lem-enhADHMeqobt01}
		AF - FA' = BF-FB' = JF = 0
		\end{equation}
		and
		\begin{equation}
		\label{eq:lem-enhADHMeqobt02}
		[A,B] + IJ = 0 \Leftrightarrow \left\{\begin{array}{l}
		[A',B'] = 0\\
		A'\widetilde{B}+ \widetilde{A}B''-B'\widetilde{A}-\widetilde{B}A''+\widetilde{I}J'' = 0\end{array}\right. .
		\end{equation}
		
The map $F$ above is clearly injective; in addition, note that the ADHM datum $(A,B,I,J)$ is stable if and only if it satisfies:
		\begin{itemize}
			\item[$(i)$] at least one of the maps $\widetilde{A}$, $\widetilde{B}$ and $\widetilde{I}$ is nontrivial;
			\item[$(ii)$] there is no proper subspace $S'\subsetneq V'$ such that
			\begin{align}
			\label{prop:minilemmainside}
			\widetilde{A}(V''),\quad \widetilde{B}(V''),\quad \widetilde{I}(W)\subset S'\quad \mbox{and}\quad A'(S), B'(S)\subset S'.
			\end{align}
		\end{itemize}
		In fact, first suppose that $(A,B,I,J)$ is stable and $\widetilde{A}=\widetilde{B}=\widetilde{I}=0$. Then 
		\begin{equation*}
		A =\left[\begin{array}{cc}
		A' & 0\\
		0  & A''
		\end{array}\right],\quad B =\left[\begin{array}{cc}
		B' & 0\\
		0  & B''
		\end{array}\right],\quad I =\left[\begin{array}{c}
		0\\
		I''
		\end{array}\right], J =\left[\begin{array}{cc}
		0 & J''
		\end{array}\right].
		\end{equation*}
		Fix $(0,v'')\in 0\oplus V''$. Thus,  
		\begin{equation*}
		A(0,v'') = \left[\begin{array}{cc}
		A' & 0\\
		0  & A''
		\end{array}\right]\left[\begin{array}{c}
		0 \\ v''\end{array}\right] = \left[\begin{array}{cc}
		0 & A''(v'')\end{array}\right] \in 0\oplus V''
		\end{equation*}
		for all $(0,v'')\in 0\oplus V''$, which means $A(0\oplus V'')\subset 0\oplus V''$. Analogously $B(0\oplus V'')\subset 0\oplus V''$. Moreover, fixing $w\in W$
		\begin{equation*}
		I(W) = \left[\begin{array}{c}
		0\\
		I''
		\end{array}\right][w] = \left[\begin{array}{c}
		0\\
		I''(w)\end{array}\right] \in 0\oplus V''
		\end{equation*}
		for all $w\in W$. Therefore 
		$$
		A(0\oplus V''),\quad B(0\oplus V''),\quad I(W)\subset 0\oplus V'',
		$$
		which is a contradiction.\\
		\indent Now suppose that there exists a proper subspace $S'\subsetneq V'$ such that the conditions \eqref{prop:minilemmainside} are satisfied. Thus, $S=S'\oplus V''\subsetneq V$ is a subspace such that $A(S)$, $B(S)$, $I(W)\subset S$. Indeed, let $(s',v'')\in S'\oplus V''$. Then
		\begin{equation*}
		A(s',v'') = \left[\begin{array}{cc}
		A' & \widetilde{A}\\
		0  & A''
		\end{array}\right]\left[\begin{array}{c}
		s' \\ v''\end{array}\right] = \left[\begin{array}{cc}
		A'(s') + \widetilde{A}(v'') & A''(v'')\end{array}\right] \in S'\oplus V''
		\end{equation*}
		for all $(s',v'')\in S'\oplus V''$, i.e., $A(S'\oplus V'')\subset S'\oplus V''$. Analogously, $B(S'\oplus V'')\subset S'\oplus V''$. Moreover, fixing $w\in W$
		\begin{equation*}
		I(W) = \left[\begin{array}{c}
		\widetilde{I}\\
		I''
		\end{array}\right][w] = \left[\begin{array}{c}
		\widetilde{I}(w)\\
		I''(w)\end{array}\right] \in S'\oplus V''
		\end{equation*}
		for all $w\in W$. Therefore 
		$$
		A(S'\oplus V''), B(S'\oplus V''), I(W)\subset S'\oplus V'',
		$$
		which contradicts the stability condition. Therefore, if $(A,B,I,J)$ is stable, it satisfies the conditions $(i)$ and $(ii)$ above.\\
		\indent Now suppose that $(A,B,I,J)$ satisfies the conditions $(i)$ and $(ii)$. One can check that $(A,B,I,J)$ is a stable data. Indeed, let $S=S'\oplus S''\subset V$ such that $A(S)$, $B(S)$, $I(W)\subset S$ and $(s',s'')\in S$. Thus,
		\begin{equation*}
		A(s',s'') = \left[\begin{array}{cc}
		A' & \widetilde{A}\\
		0  & A''
		\end{array}\right]\left[\begin{array}{c}
		s' \\ s''\end{array}\right] = \left[\begin{array}{cc}
		A'(s') + \widetilde{A}(s'') & A''(s'')\end{array}\right] \in S'\oplus S''
		\end{equation*}
		for all $(s',s'')\in S'\oplus S''$, which means that $A'(S')+\widetilde{A}(S'')\subset S'$ and $A''(S'')\subset S''$. Analogously, $B'(S')+B''(S'')\subset S'$ and $B''(S'')\subset S''$. Moreover, given $w\in W$ 
		\begin{equation*}
		I(w) = \left[\begin{array}{c}
		\widetilde{I}\\
		I''
		\end{array}\right][w] = \left[\begin{array}{c}
		\widetilde{I}(w)\\
		I''(w)
		\end{array}\right]\in S'\oplus S''
		\end{equation*}
		However, since the ADHM data $(A'',B'',I'',J'')$ is stable, $S''=V''$. Thus, $S'$ is a subspace which satisfies the conditions in \eqref{prop:minilemmainside}. It follows from $(ii)$ that $S'=0$ or $S'=V'$. If $S'=0$, $\widetilde{A}(V'')$, $\widetilde{B}(V'')$, $\widetilde{I}(W)\subset\{0\}$ and $\widetilde{A}=\widetilde{B}=\widetilde{I}=0$, which contradicts the condition $(i)$. Therefore, $S'=V'$ and $S=V$, i.e., the ADHM data $(A,B,I,J)$ is in fact stable if and only if $(A,B,I,J)$ satisfies the conditions $(i)$ and $(ii)$ above.

In order to complete the proof, it is enough to show that there exists a nontrivial solution for the equation \eqref{eq:lem-enhADHMeqobt02} which satisfies the conditions $(i)$ and $(ii)$. First choose a basis $\{v_1,\ldots,v_{c'}\}$ for $V'$ and let $A'$ and $B'$ be two diagonal matrix,
		\begin{align}
		A'=\mbox{diag}(\alpha_1,\ldots,\alpha_{c'}),\quad B'=\mbox{diag}(\beta_1,\ldots,\beta_{c'}) \nonumber,
		\end{align} 
		such that 
		\begin{equation*}
		\left\{\begin{array}{c}
		\alpha_i\neq\alpha_j\mbox{, for}i\neq j\\
		\beta_i\neq\beta_j\mbox{, for }i\neq j
		\end{array}\right. .
		\end{equation*}
		Let $\widetilde{I}:W\longrightarrow V'$ be a linear map of rank is $1$ and $Im(\widetilde{I})$ is generated by the vector 
		$$
		v = \sum_{i=1}^{c'}v_i.
		$$
		Therefore, $\{v, B'v, \ldots, B^{\prime c'-1}v\}$ is a basis for $V'$, otherwise, there would exist a nontrivial linear relation of the form 
		$$
		\sum_{i=0}^{c'-1}x_iB^{\prime i}v=0.
		$$
		Thus, for $B'=\mbox{diag}({\beta_1,\ldots,\beta_{c'}})$, $x_i's$ are a solution for the linear system
		$$
		\sum_{i=1}^{c'}\beta^i_jx_i=0\mbox{, for }j\in\{1,\ldots,c'\},
		$$
		where $B'^0=1_{V'}$. However, the discriminant of the linear system is the Vandermond determinant
		$$
		\Delta(\beta_1,\ldots,\beta_{c'}) = \prod^{c'}_{i<j}(\beta_j-\beta_i) \neq 0,
		$$
		since $\beta_i\neq\beta_j$ for all $i\neq j$. Thus, $x_i=0$, for all $i\in\{1,\ldots,c'\}$, leading to a contradiction. In conclusion, $\{v,B'v,\ldots,B^{\prime c-1}v\}$ is a basis for $V'$. In particular, there is no subspace $0\subset S' \subset V'$ preserved by $B'$ and contained in the image of $\widetilde{I}$. Analogously, there is no subspace $0\subset S' \subset V'$ preserved by $A'$ and contained in the image of $\widetilde{I}$ as well.\\
		\indent Fixing $A'$, $B'$ and $\widetilde{I}$ as above, the equation \eqref{eq:lem-enhADHMeqobt02} is a linear system with $c'(c-c')$ equations in the $2c'(c-c')$ variables $\widetilde{A}$, $\widetilde{B}$. Such system a has a $c'(c-c')$-dimensional space of solutions. Any nontrivial solution determines a stable ADHM datum $(A,B,I,J)$. 
	\end{proof}

\begin{lem} \label{connect}
$\mathcal{N}(r,c,1)$ is connected for any $r\ge1$ and $c\ge2$.
\end{lem}

\begin{proof}
According to Lemma \ref{lem:q-sobrejetiva} above, the morphism $\mathfrak{q}:\mathcal{N}(r,c,1) \to \mathcal{M}(r,c-1)$ is surjective. Since $\mathcal{M}^{st}(r,c-1)$ is an irreducible, nonsingular variety cf. \cite[Theorem 3.3]{A01-07}, it is also connected. Therefore, it is enough to argue that the fibres of $\mathfrak{q}$ are always connected.

Indeed, note in the proof of Lemma \ref{lem:q-sobrejetiva} that $\mathfrak{q}^{-1}(A'',B'',I'',J'')$ are given by morphisms 
\begin{equation*}
\widetilde{A},\mbox{ } \widetilde{B}\in Hom(V'',V'),\mbox{ } \widetilde{I}\in Hom(W,V'),\mbox{ }
\quad \mbox{and}\quad A',B'\in End(V').
\end{equation*}
satisfying the equations on the right hand side of (\ref{eq:lem-enhADHMeqobt02}), and the open conditions (i) and (ii) above equation (\ref{prop:minilemmainside}). When $\dim V'=1$, the only proper subspace of $V'$ is the trivial one, hence conditions (i) and (ii) are redundant; it follows that  
$$ \mathfrak{q}^{-1}(A'',B'',I'',J'') = (\ker L \setminus\{0\} ) \times \mathbb{C}^2 , $$
where $L$ is the linear operator $L : Hom(V'',V')^{\oplus 2} \oplus Hom(W,V') \to Hom(V'',V')$ given by
$$ L(\widetilde{A},\widetilde{B},\widetilde{I}) := 
A'\widetilde{B} + \widetilde{A}B''- B'\widetilde{A} - \widetilde{B}A'' + \widetilde{I}J'' . $$
Note that 
$$ \dim\ker L \ge 2(c-1) + r - (c-1) = c + r - 1 \ge 2 $$
so that $\ker L \setminus\{0\}$ is connected. 
\end{proof}


\section{Smoothness} \label{smoothness}

In this section we will prove that the moduli space $\mathcal{N}(r,c,1)$ is smooth and has complex dimension $(2rc-r+1)$. In other to prove this, consider the following complex
	\begin{equation}
	\label{eq:complexo-cr-c'}
	\xymatrix{\mathcal{C}(X): 
		& \txt{$End(V)$\\ $\oplus$ \\ $End(V')$} \ar[r]^-{d_0} 
		& \txt{$End(V)^{\oplus^2}$\\ $\oplus$ \\ $Hom(W,V)$\\ $\oplus$ \\ $Hom(V,W)$\\ $\oplus$ \\ $End(V')^{\oplus^2}$\\ $\oplus$ \\ $Hom (V',V)$}  \ar[r]^-{d_1}
		& \txt{$End(V)$\\ $\oplus$ \\ $Hom(V',V)^{\oplus^2}$\\ $\oplus$ \\ $Hom(V',W)$\\ $\oplus$ \\ $End(V')$} \ar[r]^-{d_2}
		& \txt{$Hom(V',V)$}
	}
	\end{equation}
	with 
	\begin{equation*}
	\begin{array}{rcl}
	d_0(h,h')                & = & ([h,A], [h,B], hI, -Jh, [h',A'], [h',B'], hF-Fh')\\
	d_1(a,b,i,j,a',b',f)     & = & ([a,B] + [A,b] + Ij + iJ, Af+aF-Fa'-fA', \\
	&   & Bf+bF-Fb'-fB', jF +Jf, [a',B']+[A',b'])\\
	d_2(c_1,c_2,c_3,c_4,c_5) & = & c_1F + Bc_2 - c_2B' +c_3A'-Ac_3-Ic_4-Fc_5.
	\end{array}
	\end{equation*}
Note that $d_0$ is the linearization of the free action \eqref{eq:acaolivre}, while $d_1$ is the linearization of the equations in \eqref{eq:quiverADHMenh}. 

\begin{thm}\label{teo:rep-ref-est-tipo-num-rcc'}
Let $\dim(W)=r$, $\dim(V)=c$ and $\dim(V')=c'$, and let $X = (V, V', A, B, I, J, A', B', F)$ be a stable enhanced ADHM datum. Then 
\begin{equation*}
H^0(\mathcal{C}(X)) = H^3(\mathcal{C}(X)) =  0,
\end{equation*}
where $\mathcal{C}(X)$ is the complex (\ref{eq:complexo-cr-c'}).
\end{thm}
\begin{proof} 
Consider the shifted complex $\mathcal{C}(X)[1]^i := \mathcal{C}(X)^{i+1}$ with $(d_{i})_{\mathcal{C}(X)[1]}:=(-1)d_{i+1}$; it is given by:
\begin{equation*} \xymatrix{
\mathcal{C}(X)[1]:
			& \txt{$End(V)^{\oplus^2}$ \\ $\oplus$\\ $Hom(W,V)$\\ $\oplus$\\ $Hom(V,W)$\\ $\oplus$\\ $End(V')^{\oplus^2}$\\ $\oplus$\\ $Hom(V',V)$} 
			\ar[r]^-{d_0}
			& \txt{$End(V)$\\ $\oplus$\\ $Hom(V',V)^{\oplus^2}$\\ $\oplus $\\ $Hom(V,W)$\\ $\oplus$\\ $End(V')$\\} \ar[r]^-{d_1}
			& \txt{$Hom(V',V)$}
} \end{equation*}
		with 
		\begin{equation*}
		\begin{array}{rcl}
		d_0(a,b,i,j,a',b',f) & = & -([a,B] + [A,b] + Ij + iJ, Af+aF-Fa'-fA', \\
		&   &Bf+bF-Fb'-fB', jF +Jf, [a',B']+[A',b'])\\
		d_1(c_1,c_2,c_3,c_4,c_5) & = & -c_1F - Bc_2 + c_2B' -c_3A'+Ac_3+Ic_4+Fc_5.
		\end{array}
		\end{equation*}
		
Consider also the following complexes:
		\begin{equation*}
		\label{eq:ca-complex-deformation}
		\xymatrix{\mathcal{C}(\mathcal{A}):
			& End(V) \ar[r]^-{d_0}
			& \txt{$End(V)^{\oplus^2}$\\ $\oplus$\\ $Hom(W,V)$\\ $\oplus$\\ $Hom(V,W)$\\}\ar[r]^-{d_1}
			& End(V)
		}
		\end{equation*}
		where 
		\begin{equation*}
		\begin{array}{rcl}
		d_0(a)       & = & ([h,A], [h,B], hI, -Jh)\\
		d_1(a,b,i,j) & = & [a,B]+[A,b]+Ij+iJ;
		\end{array}
		\end{equation*}
		
		\begin{equation}
		\label{eq:cb-complex}
		\xymatrix{\mathcal{C}(\mathcal{B}):
			& End(V') \ar[r]^-{d_0}
			& End(V')
		}
		\end{equation}
		where 
		\begin{equation*}
		\begin{array}{rcl}
		d_0(h') & = & [h',B];
		\end{array}
		\end{equation*}
		and
		\begin{equation*}
		\xymatrix{\mathcal{C}(\mathcal{A},\mathcal{B}):
			& \txt{$ Hom(V',V)$\\ $\oplus$\\ $End(V')$} \ar[r]^-{d_0}
			& \txt{$Hom(V',V)^{\oplus^2}$\\ $\oplus$\\ $Hom(V',W)$\\ $\oplus$\\ $End(V')$} \ar[r]^-{d_1}
			& Hom(V',V)}
		\end{equation*}
		where 
		\begin{equation*}
		\begin{array}{lcl}
		d_0(f,a') & = & (-Af+fA'-Fa',-Bf+fB',-Jf,[a',B'])\\
		d_1(c_2,c_3,c_4,c_5) & = & - Bc_2 + c_2B' -c_3A'+Ac_3+Ic_4+Fc_5.
		\end{array}
		\end{equation*}
Define the morphism of complexes
$$ \rho:\mathcal{C}(\mathcal{A})\oplus\mathcal{C}(\mathcal{B})\longrightarrow \mathcal{C}(\mathcal{A},\mathcal{B}) $$
by 
\begin{equation*}
\begin{array}{lcl}
\rho_0(h,h') & = & (hF-Fh', [h',A'])\\
\rho_1(a,b,i,j,b') & = & (aF, bF-Fb', jF, [A',b'])\\
\rho_2(c_1) & = & c_1F
\end{array}
\end{equation*}
We assert that the cone of the map $\rho$ is equivalent to $\mathcal{C}(X)[1]$. In fact, denote this cone by $(C,d_C)$; it follows that
		\begin{equation*}
		\left\{\begin{array}{c}
		C^i = \mathcal{C}(\mathcal{A},\mathcal{B})^i\oplus(\mathcal{C}(\mathcal{A})^{i+1}\oplus\mathcal{C}(\mathcal{B})^{i+1})\\
		(d_i)_C = ((d_i)_{\mathcal{C}(\mathcal{A},\mathcal{B})^i} -\rho_{i+1}, -(d_{i+1})_{(\mathcal{C}(\mathcal{A})^{i+1}\oplus\mathcal{C}(\mathcal{B})^{i+1})} )
		\end{array}\right. .
		\end{equation*}
		Therefore,
		\begin{equation*}
		\begin{array}{ccl}
		C^0     & = & Hom(V',V)\oplus End(V')\oplus End(V)^{\oplus^2}\oplus Hom(W,V)\oplus \\
		        &   & \oplus Hom(V,W)\oplus End(V')\\
		C^1     & = & Hom(V',V)^{\oplus^2}\oplus Hom(V,W)\oplus End(V')\oplus End(V)\\
		C^2     & = & Hom(V',V)\\
		&   & \\
		\end{array}
		\end{equation*}
		\begin{equation*}
		\begin{array}{ccl}
		(d_0)_C(f,a',a,b,i,j,b')	&	= & (d_0(f,a')-\rho_1(a,b,i,j,b'), -d_1(a,b,i,j,b'))\\
		& = & ((-Af+fA'+Fa',-Bf+fB',-Jf,-[a',B'])-\\
		& = & (aF,bF-Fb',jF,[A',b']), -[a,B]-[A,b]-Ij-iJ)\\
		& = & (-Af+fA'+Fa'-aF,-Bf+fB'-bF+Fb',\\
		&   & -Jf-jF,-[A',b']-[a,B'],-[a,B]-[A,b]-Ij-iJ)\\
		&   & \\
		(d_1)_C(c_2,c_3,c_4,c_5)	&	= & (d_1(c_2,c_3,c_4,c_5)-\rho_1(c_1), -d_2(c_1))\\
		& = & (-c_1F-Bc_2+c_2B'-c_3A'+Ac_3+Ic_4+Fc_5,0).
		\end{array}
		\end{equation*}
		Hence the cone of the map $\rho$ is equivalent to $\mathcal{C}(X[1])$. So, one can obtain the following exact triangle 
		\begin{equation}\label{eq:seqexata01}
		\xymatrix{\mathcal{C}(X) \ar[r] & \mathcal{C}(\mathcal{A})\oplus\mathcal{C}(\mathcal{B}) \ar[r]^-{\rho} & \mathcal{C}(\mathcal{A},\mathcal{B})
		}
		\end{equation}

Since $\mathcal{C}(\mathcal{A})$ is just the deformation complex for the usual ADHM equation, it is well known that 
\begin{equation}\label{eq:H0-H2-de-cA-nulo}
H^0(\mathcal{C}(\mathcal{A}))=H^2(\mathcal{C}(\mathcal{A}))=0.
\end{equation}
		Let us prove that $H^2(\mathcal{C}(\mathcal{A},\mathcal{B})) = 0$. In fact, the dual of the differential
		$$
		d_1: \mathcal{C}(\mathcal{A},\mathcal{B})^1 \longrightarrow \mathcal{C}(\mathcal{A},\mathcal{B})^2 
		$$
		is given by
		\begin{equation*}
		\begin{array}{cccl}
		d_1^{\vee}: & Hom(V,V') & \longrightarrow & Hom(V,V')^{\oplus^2}\oplus Hom(W,V')\oplus End(V')\\
		& \varphi   & \longmapsto     & (B'\varphi - \varphi B, \varphi A- A'\varphi, \varphi I, \varphi F)
		\end{array}.
		\end{equation*}
		Suppose that $d_1^{\vee}(\varphi) = 0$. Thus, 
		$$
		B'\varphi-\varphi B = A'\varphi-\varphi A = \varphi I = 0.
		$$
		Therefore, 
		$$
		Im(I), A(\ker(\varphi)), B(\ker(\varphi)) \subseteq \ker(\varphi).
		$$
		The fact that $Im(I)\subseteq\ker(\varphi)$ is trivial. If $v\in B(\ker(\varphi))$, then there exists $w\in\ker(\varphi)$ such that $v = Bw$. Then
		$$
		\varphi(v) = \varphi(B(w)) = B'(\varphi(w)) = B'(0) = 0.
		$$
		Therefore $v\in\ker(\varphi)$. Hence, $B(\ker(\varphi))\subseteq\ker(\varphi)$. One can prove that $A(\ker(\varphi))\subseteq\ker(\varphi)$ analogously. It follows from the stability of $x$ that $\ker(\varphi) = 0$ or $\ker(\varphi) = V'$. If $\ker(\varphi) = 0$, then $I\equiv 0$. This lead us to a contradiction. Therefore, $\varphi = 0$, i.e., $d_1^{\vee}$ it is injective and hence $d_1$ it is surjective. It follows that 
\begin{equation*}
H^2(\mathcal{C}(\mathcal{A},\mathcal{B})) = Hom(V',V)/ Im(d_1) = 0
\end{equation*}

Next, let us prove that $H^0(\rho)$ is injective. Since $H^0(\mathcal{C}(\mathcal{A}))=0$,  we have that 
\begin{equation*}
H^0(\mathcal{C}(\mathcal{A})\oplus\mathcal{C}(\mathcal{B})) = H^0(\mathcal{C}(\mathcal{B})).
\end{equation*}
Therefore,
		\begin{equation*}
		\begin{array}{cccl}
		H^0(\rho): & H^0(\mathcal{C}(\mathcal{B})) &\longrightarrow & H^0(\mathcal{C}(\mathcal{A},\mathcal{B})) \\\
		& \overline{h'} & \longmapsto & (\overline{-Fh'},\overline{[h',A]})
		\end{array}
		\end{equation*}
		where $\overline{x}\in H^0(\mathcal{C})$ denotes the equivalence class of $x\in \mathcal{C}$, with $\mathcal{C}\in\{\mathcal{C}(\mathcal{A}), \mathcal{C}(\mathcal{B}), \mathcal{C}(\mathcal{A},\mathcal{B})\}$.
		Suppose that $H^0(\rho)(\overline{h'}) = 0$. Then, $\overline{-Fh'}=0$. Since $F$ it is injective, it is true that $\overline{h'}=\overline{0}$. Hence $H^0\rho$ it is injective.\\
		
Finally, it follows from equation (\ref{eq:H0-H2-de-cA-nulo}) that the exact sequence of cohomologies of (\ref{eq:seqexata01}) is given by 
		\begin{equation}\label{eq:seqexatalonga01}
		\xymatrix{
			0\ar[r] & H^0(\mathcal{C}(R)) \ar[r]^-{\delta} & H^0(\mathcal{C}(B)) \ar[r]^-{H^0(\rho)} & H^0(\mathcal{C}(\mathcal{A},\mathcal{B})) \ar[r] & \ldots\\
			\ldots \ar[r] &\txt{$\underbrace{H^2(\mathcal{C}(B))}_{=0}$} \ar[r] & \txt{$\underbrace{H^2(\mathcal{C}(\mathcal{A},\mathcal{B}))}_{=0}$} \ar[r]^-{\gamma} & H^3(\mathcal{C}(R)) \ar[r] & 0.
		}
		\end{equation}
The claim in the Theorem follows immediately.
	\end{proof}

\begin{obs} \label{rmk about h2}
Note that the exact sequence of cohomologies (\ref{eq:seqexata01}) reduces to
$$ 0 \to H^0(\mathcal{C}(\mathcal{B})) \stackrel{H^0(\rho)}{\longrightarrow} H^0(\mathcal{C}(\mathcal{A},\mathcal{B})) \to
H^1(\mathcal{C}(X)) \to $$
$$ \to H^1(\mathcal{C}(\mathcal{A})\oplus \mathcal{C}(\mathcal{B})) \stackrel{H^1(\rho)}{\longrightarrow}
H^1(\mathcal{C}(\mathcal{A},\mathcal{B})) \to H^2(\mathcal{C}(X)) \to 0 $$
The cohomology group $H^2(\mathcal{C}(X))$ measures the obstruction the smoothness of the variety $\mathcal{N}(r,c,c')$, whose Zariski tangent space is precisely $H^1(\mathcal{C}(X))$.
\end{obs}

We now focus on the particular case $c'=\dim(V')=1$. The equation $[A',B'] = 0$ becomes vacuous, hence it can omit this equation from the set of enhanced ADHM equation, obtaining the equations
\begin{equation*} \begin{array}{ccccc}
\label{eq:ADHMenh02}
[A,B]+IJ = 0, & AF-FA'=0, & BF-FB'=0, &JF=0.
\end{array} \end{equation*}
The deformation complex (\ref{eq:complexo-cr-c'}) simplifies to the following
\begin{equation}\label{def-cpx2}
\xymatrix{ \mathcal{C}(X): 
& \txt{$End(V)$\\ $\oplus$ \\ $End(V')$} \ar[r]^-{d_0} 
& \txt{$End(V)^{\oplus^2}$\\ $\oplus$ \\ $Hom(W,V)$\\ $\oplus$ \\ $Hom(V,W)$\\ $\oplus$ \\ $End(V')^{\oplus^2}$\\ $\oplus$ \\ $Hom (V',V)$}  \ar[r]^-{d_1}
& \txt{$End(V)$\\ $\oplus$ \\ $Hom(V',V)^{\oplus^2}$\\ $\oplus$ \\ $Hom(V',W)$} \ar[r]^-{d_2}
& \txt{$Hom(V',V)$}
} \end{equation}
with maps given by
\begin{equation*} \begin{array}{rcl}
d_0(h,h')                & = & ([h,A], [h,B], hI, -Jh, [h',A'], [h',B'], aF-Fa')\\
d_1(a,b,i,j,a',b',f)     & = & ([a,B] + [A,b] + Ij + iJ, Af+aF-Fa'-fA', \\
&   & Bf+bF-Fb'-fB', jF +Jf)\\
d_2(c_1,c_2,c_3,c_4) & = & c_1F + Bc_2 - c_2B' +c_3A'-Ac_3-Ic_4.
\end{array} \end{equation*}

We are finally in position to prove the main Theorem of this section.

\begin{thm} \label{teo:N-eh-suave}
 Let $\dim(W) = r$, $\dim(V)=c$ and $\dim(V')=1$, and let $X = (A, B, I, J, A', B', F)$ be a stable enhanced ADHM datum which satisfies the enhanced ADHM equations. The moduli space $\mathcal{N}(r,c,1)$ is a non-singular, quasi-projective variety of dimension is $2rc-r+1$. 
\end{thm}
	
\begin{proof}
Let $\mathcal{C}(\mathcal{A})$ and $\mathcal{C}(\mathcal{B})$ be the complexes (\ref{eq:ca-complex-deformation}) and (\ref{eq:cb-complex}), respectively. Let $\mathcal{C}(\mathcal{A},\mathcal{B})$ be the complex given by
\begin{equation*}
\xymatrix{\mathcal{C}(\mathcal{A},\mathcal{B}):
& \txt{$ Hom(V',V)$\\ $\oplus$\\ $End(V')$} \ar[r]^-{d_0}
& \txt{$Hom(V',V)^{\oplus^2}$\\ $\oplus$\\ $Hom(V',W)$} \ar[r]^-{d_1}
& Hom(V',V)}
\end{equation*}
with maps given by
\begin{equation*} \begin{array}{lcl}
d_0(f,a') & = & (-Af+fA'-Fa',-Bf+fB',-Jf)\\
d_1(c_2,c_3,c_4) & = & - Bc_2 + c_2B' -c_3A'+Ac_3+Ic_4.
\end{array} \end{equation*}
Consider the morphism of complexes:
$$ \rho:\mathcal{C}(\mathcal{A})\oplus\mathcal{C}(\mathcal{B})\longrightarrow \mathcal{C}(\mathcal{A},\mathcal{B}) $$
given by
\begin{equation*}	\begin{array}{lcl}
\rho_0(a,a') & = & (aF-Fa', [a',A'])\\
\rho_1(a,b,i,j,b') & = & (aF, bF-Fb', jF, 0)\\
\rho_2(c_1) & = & c_1F
\end{array}. \end{equation*}

One can check that the sequence of complexes
\begin{equation}\label{eq:seqexata02}
\xymatrix{ 0\ar[r] & \mathcal{C}(R)\ar[r]& \mathcal{C}(\mathcal{A})\oplus\mathcal{C}(\mathcal{B}) \ar[r]^-{\rho} 
&\mathcal{C}(\mathcal{A},\mathcal{B}) \ar[r] & 0	}
\end{equation}
is exact. Similarly to Theorem \ref{teo:rep-ref-est-tipo-num-rcc'}, one can prove that
$$ H^0(\mathcal{C}(X)) = H^3(\mathcal{C}(X)) = 0, $$ 
therefore the exact sequence of cohomologies associated to (\ref{eq:seqexata02}) reduces to
$$ 0 \to H^0(\mathcal{C}(\mathcal{B})) \stackrel{H^0(\rho)}{\longrightarrow} H^0(\mathcal{C}(\mathcal{A},\mathcal{B})) \to
H^1(\mathcal{C}(X)) \to $$
$$ \to H^1(\mathcal{C}(\mathcal{A})\oplus \mathcal{C}(\mathcal{B})) \stackrel{H^1(\rho)}{\longrightarrow}
H^1(\mathcal{C}(\mathcal{A},\mathcal{B})) \to H^2(\mathcal{C}(X)) \to 0 $$
To complete the proof, it remains to show that $H^1(\rho)$ is a surjective map, which implies that the obstruction space $H^2(\mathcal{C}(R))$ vanishes. 

In order to establish the surjectivity of $H^1(\rho)$, we show that the following induced map is surjective, 
\begin{equation} \label{eq:z1-rho}
Z^1(\rho): Z^1(\mathcal{C}(\mathcal{A}))\oplus Z^1(\mathcal{C}(\mathcal{B})) \longrightarrow Z^1(\mathcal{C}(\mathcal{A},\mathcal{B})),
\end{equation}
where $Z^1(\mathcal{C}(\mathcal{A})):=\ker(d_1)$, with $d_1:\mathcal{C}(\mathcal{A})^1\longrightarrow\mathcal{C}(\mathcal{A})^2$, and similarly for $Z^1(\mathcal{C}(\mathcal{B}))$ and for $Z^1(\mathcal{C}(\mathcal{A},\mathcal{B}))$. Indeed, note that $Z^1(\mathcal{C}(\mathcal{B}))=0$; our strategy is to show that:
\begin{enumerate}
\item the diagram
			\begin{equation*}
			\label{eq:grafico_z1-rho}
			\xymatrix{
				0 \ar[r]  & 
				*\txt{ $Z^1(\mathcal{C}(\mathcal{A}))\oplus Z^1(\mathcal{C}(\mathcal{B}))$ } \ar[dd]^-{\txt{ $Z^1(\rho)$ }} \ar[r]^-{i} &
				*\txt{ $\mathcal{C}(\mathcal{A})^1\oplus\mathcal{C}(\mathcal{B})^1$ } \ar[r]^-{d_1} \ar[dd]^-{\txt{ $\rho_1$ }} &
				*\txt{ $\overbrace{\mathcal{C}(\mathcal{A})^2}^{=End(V)}$ } \ar[r] \ar[dd]^-{\txt{ $\rho_2$ }} & 0\\
				& & & &\\
				0 \ar[r] & 
				*\txt{ $Z^1(\mathcal{C}(\mathcal{A},\mathcal{B}))$ } \ar[r]^-{i} & 
				*\txt{ $\mathcal{C}(\mathcal{A},\mathcal{B})^1$ } \ar[r]^-{d_1} &
				\txt{ $\underbrace{\mathcal{C}(\mathcal{A},\mathcal{B})^2}_{=Hom(V',V)}$ } \ar[r] & 0
			}
			\end{equation*}
			commutes;
			\item the maps $\rho_1$ and $\rho_2$ are surjective;
			\item for all $p\in Z^1(\mathcal{C}(\mathcal{A},\mathcal{B}))$, $\rho^{-1}_1(p)\cap (Z^1(\mathcal{C}(\mathcal{A}))\oplus Z^1(\mathcal{C}(\mathcal{B})))\neq 0$ in $\mathcal{C}(\mathcal{A})^1\oplus\mathcal{C}(\mathcal{B})^1.$
		\end{enumerate}
		
Proof of item 1: The map $Z^1(\rho)$ it is well-defined. Indeed, let $(a,b,i,j,b')\in\ker((d_1)_{\mathcal{C}(\mathcal{A})\oplus\mathcal{C}(\mathcal{B})})$. Thus, 
		$$
		Ij = -[a,B] - [A,b] -iJ.
		$$
		Therefore, 
		\begin{equation*}
		\begin{array}{ccl}
		d_1(\rho_1)(a,b,i,j,b') & = & d_1(aF, bF-Fb', jF, 0)\\
		& = & -BaF +aFB' - bFA' + Fb'A' + AbF - AFb' + IjF\\
		& = & [a,B]F + [A,b]F + iJF +IjF - F[A',b']\\
		& = & ([a,B] + [A,b] + iJ)F + IjF\\
		& = & IjF - IjF = 0.
		\end{array}
		\end{equation*}
		Hence the map $Z^1(\rho)$ is well-defined. Since $\rho$ is a morphism, it follows that $d_1\circ\rho_1 = \rho_2\circ d_1$. Moreover, $i\circ Z^1(\rho)=i\circ\rho_1$.\\

\medskip

Proof of item 2: Let $p=(c_2,c_3,c_4,0)\in\mathcal{C}(\mathcal{A},\mathcal{B})^1$. Let $E:V\longrightarrow V'$, such that $EF=Id_{V'}$. Since F is injective, there exists a surjective map $E$. Thus,
		\begin{equation*}
		\begin{array}{ccl}
		\rho_1(c_2E,c_3E,i,c_4E,0) & = & (c_2EF, c_3EF - F0, c_4EF, 0)\\
		& = & (c_2,c_3,c_4,0)
		\end{array}.
		\end{equation*}
		Then, $\rho_1$ is surjective. In order to show that $\rho_2$ is surjective, consider $c_1\in Hom(V',V)$. Therefore
		\begin{equation*}
		\rho_2(c_1E) =  c_1EF = c_1
		\end{equation*}
		Hence $\rho_2$ is surjective.
	
\medskip
	
\noindent Proof of item 3: Since $\rho_1$ is surjective, $\rho_1^{-1}(p)$ is a fiber over the linear space $\ker(\rho_1)$ for all $p\in\mathcal{C}(\mathcal{A},\mathcal{B})^{1}$. Since $Z^1(\mathcal{C}(\mathcal{A}))\oplus Z^1(\mathcal{C}(\mathcal{B}))$ it is a proper subspace $\mathcal{C}(\mathcal{A})^1\oplus\mathcal{C}(\mathcal{B})^1$, it remains to show that 
		$$
		\Delta = \dim(\ker(\rho_1)) + \dim(Z^1(\mathcal{C}(\mathcal{A}))\oplus Z^1(\mathcal{C}(\mathcal{B})))-\dim(\mathcal{C}(\mathcal{A})^1\oplus\mathcal{C}(\mathcal{B})^1)\geq 0.
		$$
		Indeed,
		$$
		\Delta = \dim(\ker(\rho_1)) - \dim((\mathcal{C}(\mathcal{A})^1\oplus\mathcal{C}(\mathcal{B})^1)\backslash(Z^1(\mathcal{C}(\mathcal{A}))\oplus Z^1(\mathcal{C}(\mathcal{B})))
		$$
		and $\Delta\geq 0$ means that $\ker(\rho_1)\cap (Z^1(\mathcal{C}(\mathcal{A}))\oplus Z^1(\mathcal{C}(\mathcal{B}))\neq 0$ and this conclude the proof, since the diagram above commutes. Indeed,
		\begin{equation*}
		\begin{array}{rcl}			
		\Delta & =    & \dim(\mathcal{C}(\mathcal{A})^1\oplus\mathcal{C}(\mathcal{B})^1) - \dim(Im(\rho_1)) + \dim(\mathcal{C}(\mathcal{A})^1\oplus\mathcal{C}(\mathcal{B})^1) - \\
		&      & \dim(Im(d_1)) - \dim(\mathcal{C}(\mathcal{A})^1\oplus\mathcal{C}(\mathcal{B})^1) \\
		& =    &  \dim(\mathcal{C}(\mathcal{A})^1\oplus\mathcal{C}(\mathcal{B})^1) - \dim(Im(\rho_1)) - \dim(Im(d_1))\\
		& =    & 2c^2 + 2rc + 1 - 2c - r - 1 - c^2 = c^2 + 2c(r-1) - r \ge 0
		\end{array}
		\end{equation*}
		At last, the dimension of the moduli space is equal to the dimension of $H^1(\mathcal{C}(X))$:
		\begin{equation*}
		\begin{array}{rcl}
		\dim(H^1(\mathcal{C}(X))) & = & -\dim(C(X))^0 +\dim(C(X))^1 -\dim(C(X))^2 +\dim(C(X))^3 \\
		& = & -(c^2+c'^2)  +(2c^2+2c'^2+2rc+cc') -(c^2+2c'c+c'r) +cc' \\
		& = & 2rc + c'^2-c'r = 2rc-r+1.
		\end{array}
		\end{equation*}
	\end{proof}


\section{ADHM construction of framed flags of sheaves} \label{adhm const}

In this section adapt the construction in \cite[Section 2]{nakajima}, and prove that the moduli space of framed flags of sheaves, introduced in Section \ref{sec.flags} above, is isomorphic to the moduli space of framed stable representations of the enhanced ADHM quiver. More precisely, we prove the following result.

\begin{thm} \label{thm-isom}
There exists an isomorphism of schemes $\mathcal{N}(r,n+l,l)\to\mathcal{F}(r,n,l)$, where $r,n,l\ge1$.
\end{thm}

We begin by revising a few facts about framed torsion free sheaves and stable representations of the ADHM quiver; the main references are \cite{A01-07} and \cite[Chapter 2]{nakajima}. Let $(x:y:z)$ be fixed homogeneous coordinates in $\mathbb{P}^2$, and let $\ell_\infty:=\{z=0\}$. Let $X=(W,V,A,B,I,J)$ be a representation of the ADHM quiver. The \textit{ADHM complex} associated to $X$ is the complex of locally free sheaves on $\mathbb{P}^2$ of the form
\begin{equation} \label{adhm cpx} \begin{tikzcd}
E^{\bullet}_X: V\otimes\mathcal{O}_{\mathbb{P}^2}(-1)\arrow{r}{\alpha} & (V\oplus V\oplus W)\otimes\mathcal{O}_{\mathbb{P}^2}\arrow{r}{\beta} & V\otimes\mathcal{O}_{\mathbb{P}^2}(1)
\end{tikzcd} \end{equation}
where
\begin{equation}
\alpha = \left[\begin{array}{c}	zA+x1_{V} \\ zB + y1_{V} \\ zJ
\end{array}\right],\quad 
\beta = \left[\begin{array}{ccc}-zB-y1_{V} & zA+x1_{V} & zI \end{array}\right].
\end{equation}
One can show that $\alpha$ is always injective, while $\beta$ is surjective if and only if $X$ is stable; in this case, the sheaf 
$$ E:=\mathcal{H}^0(E^{\bullet}_X)=\ker\beta/\im\alpha $$
is called the cohomology of the complex $E^{\bullet}_X$. In addition such $E$ is a rank $r$ torsion free sheaf on $\mathbb{P}^2$ with $c_2(E)=c$, framed by the induced isomorphism $\varphi:E|_{\ell_\infty}\stackrel{\sim}{\rightarrow} W\otimes\mathcal{O}_{\ell_\infty}$.

Conversely, given a framed torsion free sheaf $(E,\varphi)$ on $\mathbb{P}^2$, there is a stable ADHM datum $X$ such that $E$ is the cohomology of the the corresponding ADHM complex $E^{\bullet}_X$. This construction provides bijection between framed torsion free sheaves and stable representations of the ADHM quiver. For further references, we state the following result, cf. \cite[Chapter 2]{nakajima}.

\begin{thm} \label{teo:bijection-bet-torsion-sheaves-and-ADHM-rep}
There exists an isomorphism between the moduli space of framed torsion free sheaves on $\mathbb{P}^2$ with rank $r$ and second Chern class $n$ and the moduli space $\mathcal{M}(r,n)$ of stable representations of the ADHM quiver with numerical type $(r,n)$.
\end{thm}

We will need one more before steping into the proof of Theorem \ref{thm-isom}. Let $\mathfrak{A}$ denote the category of representations of the ADHM quiver, and let $\Kom(\mathbb{P}^2)$ be the category of complexes of sheaves on $\mathbb{P}^2$. Note that given a morphism $(\xi_1,\xi_2)$ between two representations $X$ and $\widetilde{X}$, one has the following morphism $\xi^{\bullet}=(\xi_1\oplus 1_V,(\xi_1\oplus\xi_1\oplus\xi_2)\otimes 1_V,\xi_1\otimes 1_V)$ between the ADHM complexes $E^{\bullet}_{X}$ and $E^{\bullet}_{\widetilde{X}}$
\begin{equation*}
\begin{tikzcd}
V\otimes\mathcal{O}_{\mathbb{P}^2}(-1)\arrow{r}{\alpha} \arrow{d}{\xi_1\oplus 1_V} & (V\oplus V\oplus W)\otimes\mathcal{O}_{\mathbb{P}^2}\arrow{r}{\beta}\arrow{d}{(\xi_1\oplus\xi_1\oplus\xi_2)\otimes 1_V} & V\otimes\mathcal{O}_{\mathbb{P}^2}(1)\arrow{d}{\xi_1\otimes 1_V}\\
V\otimes\mathcal{O}_{\mathbb{P}^2}(-1)\arrow{r}{\widetilde{\alpha}} & (V\oplus V\oplus W)\otimes\mathcal{O}_{\mathbb{P}^2}\arrow{r}{\widetilde{\beta}} & V\otimes\mathcal{O}_{\mathbb{P}^2}(1)
\end{tikzcd}
\end{equation*}

\begin{pps}\label{prop:funtor-exato-pleno-fiel}
The functor 
\begin{equation*}
\mathbb{K}: \mathfrak{A} \longrightarrow  \Kom(\mathbb{P}^2)
\end{equation*}
given by
\begin{align*}
\mathbb{K}(X) = E^{\bullet}_{X},\qquad \mathbb{K}(\xi_1,\xi_2) = \xi^{\bullet}.
\end{align*}
is exact, full and faithful.
\end{pps}

The proof is a straightforward exercise; it can be found in \cite{patricia}.

Now, the first step in the proof of Theorem \ref{thm-isom} is to provide a bijection between $\mathcal{N}(r,n+l,l)$ and $\mathcal{F}(r,n,l)$. So let $X=(A,B,I,J,A',B',F)$ be a stable representation of the enhanced ADHM quiver of numerical type $(r,n+l,l)$, so that $F$ is injective, and the ADHM datum $(A,B,I,J)$ is stable. By  Lemma \ref{lem:q-sobrejetiva}, one can obtain a stable representation of the ADHM quiver $X''=(A'',B'',I'',J'')$ of numerical type $(r,n)$ fitting into the following diagram

\begin{equation} \label{diag:ADHM-sequence}
\begin{tikzpicture}[->,>=stealth',shorten >=1pt,auto,node distance=2.5cm,semithick]
\tikzstyle{every state}=[fill=white,draw=none,text=black]

\node[state]    (A)                 {$V$};
\node[state]    (B) [left of=A]     {$V'$};
\node[state]    (C) [below of=A]    {$W$};
\node[state]    (D) [below of=B]    {$\{0\}$};
\node[state]    (E) [right of=A]    {$V''$};
\node[state]    (F) [below of=E]    {$W.$};

\path (A) edge [out=30, in=60, loop]           node [above] {$B$}     (A)
edge [out=120, in=150, loop]         node [above] {$A$}     (A)
edge [out=-70,in=70]                 node         {$J$}     (C)
edge [out=0,in=180]                  node         {}        (E)
(B) edge [out=30, in=60, loop]           node [above] {$B'$}    (B)
edge [out=120, in=150, loop]         node [above] {$A'$}    (B)
edge [out=0, in=180,]                node         {$F$}     (A)
edge [out=-70,in=70]                 node         {}        (D)
(C) edge [out=110,in=-110]               node         {$I$}     (A)
edge [out=0,in=180]                  node         {}        (F)
(D) edge [out=110,in=-110]               node         {}        (B)
edge [out=0,in=180]                  node         {}        (C)
(E) edge [out=30, in=60, loop]           node [above] {$B''$}   (E)
edge [out=120, in=150, loop]         node [above] {$A''$}   (E)
edge [out=-70,in=70]                 node         {$J''$}   (F)
(F) edge [out=110,in=-110]               node         {$I''$}   (E);
\end{tikzpicture} \end{equation}

Denote by $\textbf{Z}$, $\textbf{S}$ and $\textbf{Q}$ the following representations of the ADHM quiver
\begin{equation*}
\begin{tikzpicture}[->,>=stealth',shorten >=1pt,auto,node distance=2.5cm,semithick]
\tikzstyle{every state}=[fill=white,draw=none,text=black]

\node[state]    (B) [left of=A]     {$V'$};
\node[state]    (D) [below of=B]    {$\{0\}$};

\path 
(B) edge [out=30, in=60, loop]           node [above] {$B'$}    (B)
edge [out=120, in=150, loop]         node [above] {$A'$}    (B)
edge [out=-70,in=70]                 node         {\qquad ,}        (D)
(D) edge [out=110,in=-110]               node         {}        (B);
\end{tikzpicture} \qquad\begin{tikzpicture}[->,>=stealth',shorten >=1pt,auto,node distance=2.5cm,
semithick]
\tikzstyle{every state}=[fill=white,draw=none,text=black]

\node[state]    (A)                 {$V$};
\node[state]    (C) [below of=A]    {$W$};

\path (A) edge [out=30, in=60, loop]           node [above] {$B$}     (A)
edge [out=120, in=150, loop]         node [above] {$A$}     (A)
edge [out=-70,in=70]                 node         {$J$\qquad and}     (C)
(C) edge [out=110,in=-110]               node         {$I$}     (A);
\end{tikzpicture}\qquad\begin{tikzpicture}[->,>=stealth',shorten >=1pt,auto,node distance=2.5cm,
semithick]
\tikzstyle{every state}=[fill=white,draw=none,text=black]

\node[state]    (E) [right of=A]    {$V''$};
\node[state]    (F) [below of=E]    {$W.$};

\path 
(E) edge [out=30, in=60, loop]           node [above] {$B''$}   (E)
edge [out=120, in=150, loop]         node [above] {$A''$}   (E)
edge [out=-70,in=70]                 node         {$J''$\qquad,}   (F)
(F) edge [out=110,in=-110]               node         {$I''$}   (E);
\end{tikzpicture}
\end{equation*}
respectively. Since $F$ is injective, and the map between $V$ and $V''$ is, by construction, surjective, the previous diagram \eqref{diag:ADHM-sequence} can be expressed as a short exact sequence in $\mathfrak{A}$:
\begin{equation} \label{eq:seq-exata-curta-ZSQ} 
0 \to \textbf{Z} \to \textbf{S} \to \textbf{Q} \to 0
\end{equation}

It then follows from Proposition \ref{prop:funtor-exato-pleno-fiel} that
\begin{equation}\label{cpx's}
0 \to E^{\bullet}_{\textbf{Z}} \to E^{\bullet}_{\textbf{S}} \to E^{\bullet}_{\textbf{Q}} \to 0,
\end{equation}
where $E^{\bullet}_{\textbf{Z}}$, $E^{\bullet}_{\textbf{S}}$ and $E^{\bullet}_{\textbf{Q}}$ are the ADHM complexes on $\mathbb{P}^2$ associated with the representations of the ADHM quiver $\textbf{Z}$, $\textbf{S}$ and $\textbf{Q}$, respectively, is a short exact sequence on $\Kom(\mathbb{P}^2)$. 

Since $\textbf{S}$ and $\textbf{Q}$ are stable, it follows that $\mathcal{H}^{p}(E^{\bullet}_{\textbf{S}})=
\mathcal{H}^{p}(E^{\bullet}_{\textbf{Q}})=0$ for $p=-1,1$. Since the morphism $\alpha$ in the complex 
$E^{\bullet}_{\textbf{Z}}$ is injective, we conclude that $\mathcal{H}^{-1}(E^{\bullet}_{\textbf{Z}})=0$; in addition, one can also check that $\mathcal{H}^{0}(E^{\bullet}_{\textbf{Z}})=0$, see the proof of \cite[Thm. 5.5]{A01-07}. These facts imply that the long exact sequence in cohomology associated with sequence (\ref{cpx's}) reduces to
\begin{equation} \label{eq:seq-exata-flag-of-sheaves} \begin{tikzcd}
0 \arrow{r}              & \mathcal{H}^{0}(E^{\bullet}_{\textbf{S}}) \arrow{r}  & \mathcal{H}^{0}(E^{\bullet}_{\textbf{Q}}) \arrow{r}  & \mathcal{\mathcal{H}}^{1}(E^{\bullet}_{\textbf{Z}}) \arrow{r}  & 0.
\end{tikzcd} \end{equation}
By the usual ADHM construction, outline above, the pair $(\mathcal{H}^{0}(E^{\bullet}_{\textbf{Q}}),\varphi)$ is a rank $r$ framed torsion free sheaf with second Chern class $n$; $\mathcal{H}^{0}(E^{\bullet}_{\textbf{S}})$ is a subsheaf of $\mathcal{H}^{0}(E^{\bullet}_{\textbf{Q}})$; finally, the quotient sheaf 
$$ \mathcal{H}^{1}(E^{\bullet}_{\textbf{Z}}) \cong 
\mathcal{H}^{0}(E^{\bullet}_{\textbf{Q}})/\mathcal{H}^{0}(E^{\bullet}_{\textbf{S}}) $$
is a 0-dimensional sheaf of length $l$ supported outside $\ell_\infty$, see the proof of \cite[Thm. 5.5]{A01-07}.

Summarizing, the triple $(\mathcal{H}^{0}(E^{\bullet}_{\textbf{S}}),\mathcal{H}^{0}(E^{\bullet}_{\textbf{Q}}),\varphi)$ is a framed flag of sheaves, and it yields a point of $\mathcal{F}(r,n,l)$.

Now let that $(E,F,\varphi)$ be a framed flag of sheaves with numerical invariants $r:={\rm rk}(E)={\rm rk}(F)$, $n:=c_2(F)$ and $l:=h^0(F/E)$. Since $(F,\varphi)$ and $(E,\varphi)$ are framed torsion free sheaves of rank $r$ and second Chern classes $n$ and $n+l$, respectively, there are stable representations of the ADHM quiver $\textbf{Q}=(W'',V'',A'',B'',I'',J'')$ of numerical type $(r,n)$, and $\textbf{S}=(W,V,A,B,I,J)$ with numerical type $(r,n+l)$ associated with $(F,\varphi)$ and $(E,\varphi)$, respectively. By Proposition \ref{prop:funtor-exato-pleno-fiel}, the inclusion $E\hookrightarrow F$ provides a morphism $\Psi:\textbf{S}\longrightarrow\textbf{Q}$ between representations of the ADHM quiver; we check that $\Psi$ is an epimorphism.

Indeed, recall from the proof of \cite[Thm. 2.1]{nakajima} that the vector spaces $V$ and $W$ in the ADHM complex (\ref{adhm cpx}) associated with a framed sheaf $(F,\varphi)$ are given by
$$ V \simeq H^1(F(-1)),\quad W \simeq H^0(F|_{l_{\infty}}) , $$
and similarly for the framed sheaf $(F,\varphi)$. Since the quotient $F/E$ is a 0-dimensional sheaf supported away from $\ell_\infty$, we obtain an isomorphism 
$$ \Psi_2: H^0(F|_{l_{\infty}}) \stackrel{\sim}{\rightarrow} H^0(F|_{l_{\infty}}) . $$
From the short exact sequence $0\to E \to F \to F/E \to 0$, we obtain the following exact sequence in cohomology
$$ 0 \to H^0(F/E(-1)) \to H^1(E(-1)) \stackrel{\Psi_1}{\longrightarrow} \to H^1(F(-1)) \to 0 , $$
since $H^0(F(-1))=0$ ($F$ is a $\mu$-semistable sheaf with $c_1(F)=0$) and $H^1(F/E(-1))=0$ (because $F/E$ is a 0-dimensional sheaf). The morphism $\Psi:\textbf{S}\longrightarrow\textbf{Q}$ is precisely given by the pair $(\Psi_1,\Psi_2)$; since both maps are surjective, it follows that $\Psi$ is an epimorphism, as desired.

The short exact sequence
$$ 0 \to \ker\Psi \to \textbf{S} \stackrel{\Psi}{\longrightarrow} \textbf{Q} \to 0 $$
translates into the diagram 
\begin{equation} \label{diag:ADHM-sequence02}
\begin{tikzpicture}[->,>=stealth',shorten >=1pt,auto,node distance=2.5cm,semithick]
\tikzstyle{every state}=[fill=white,draw=none,text=black]

\node[state]    (A)                 {$V$};
\node[state]    (B) [left of=A]     {$V'$};
\node[state]    (C) [below of=A]    {$W$};
\node[state]    (D) [below of=B]    {$\{0\}$};
\node[state]    (E) [right of=A]    {$V''$};
\node[state]    (F) [below of=E]    {$W$};

\path (A) edge [out=30, in=60, loop]           node [above] {$B$}             (A)
edge [out=120, in=150, loop]         node [above] {$A$}             (A)
edge [out=-70,in=70]                 node         {$J$}             (C)
edge [out=0,in=180]                  node         {$\Psi_1$}        (E)
(B) edge [out=30, in=60, loop]           node [above] {$B'$}    (B)
edge [out=120, in=150, loop]         node [above] {$A'$}    (B)
edge [out=0, in=180,]                node         {$F$}     (A)
edge [out=-70,in=70]                 node         {}        (D)
(C) edge [out=110,in=-110]               node         {$I$}     (A)
edge [out=0,in=180]                  node         {$\Psi_2$}        (F)
(D) edge [out=110,in=-110]               node         {}        (B)
edge [out=0,in=180]                  node         {}        (C)
(E) edge [out=30, in=60, loop]           node [above] {$B''$}   (E)
edge [out=120, in=150, loop]         node [above] {$A''$}   (E)
edge [out=-70,in=70]                 node         {$J''\quad .$}   (F)
(F) edge [out=110,in=-110]               node         {$I''$}   (E);
\end{tikzpicture}
\end{equation}
It follows that 
$$ X=(W,V,V',A,B,I,J,A',B',F) $$ 
is a framed stable representation of the enhanced ADHM quiver. Furthermore, the cohomology groups of the monads $E^{\bullet}_{\textbf{Z}}$, $E^{\bullet}_{\textbf{S}}$ and $E^{\bullet}_{\textbf{Q}}$ associated with the representations of the ADHM quiver ${\textbf{Z}}:=\ker\Psi$, \textbf{S} and \textbf{Q}, respectively, lead to the exact sequence in (\eqref{eq:seq-exata-flag-of-sheaves}). This completes the description of a bijection $\mathbf{b}:\mathcal{N}(r,n+l,l)\to\mathcal{F}(r,n,l)$.

\bigskip

In order to finalize the proof of Theorem \ref{thm-isom}, note that the complex of sheaves on $\p2$ in equation (\ref{adhm cpx}) can also be regarded as a family of complexes parametrized by $\mathcal{M}(r,c)$, i.e. a complex of sheaves on $\p2\times\mathcal{M}(r,c)$. Similarly, diagram (\ref{diag:ADHM-sequence}) allows as to think of equation (\ref{cpx's}) as a short exact sequence of complexes parametrized by $\mathcal{N}(r,c,c')$. Passing to cohomology as in equation (\ref{eq:seq-exata-flag-of-sheaves}), we obtain a family of framed flags of sheaves parametrized by $\mathcal{N}(r,c,c')$; in other words, the bijection $\mathbf{b}:\mathcal{N}(r,n+l,l)\to\mathcal{F}(r,n,l)$ just described can be regarded as an element of the set $\mathbf{F}_{(r,n,l)}(\mathcal{N}(r,n+l,l))$, described in the beginning of Section \ref{sec.flags} above. By the representability of the moduli functor $\mathbf{F}_{(r,n,l)}$, there is an unique bijective morphism of schemes $\mathfrak{b}:\mathcal{N}(r,n+l,l)\to\mathcal{F}(r,n,l)$ corresponding to the bijection $\mathbf{b}$.

On the other hand, as it happens with the usual ADHM construction, one can follow the steps of the inverse of the bijection $\mathbf{b}$ and note that any family of framed flags of sheaves $(F_S,\varphi_S,Q_S,g_S)$ parametrized by a scheme $S$ gives rise to a family of representations of the enhanced ADHM quiver parametrized by $S$, thus a morphism $S\to\mathcal{N}(r,n+l,l)$. This means that we get a map of sets 
$$ \mathbf{F}_{(r,n,l)}(S) = \Hom(S,\mathcal{F}(r,n,l)) \to \Hom(S,\mathcal{N}(r,n+l,l)) . $$
Applying this principle to the universal family, that is the family corresponding to the identity morphism in $\Hom(\mathcal{F}(r,n,l),\mathcal{F}(r,n,l))$, we obtain a morphism of schemes $\mathcal{F}(r,n,l) \to \mathcal{N}(r,n+l,l)$ which is the inverse of $\mathfrak{b}$.

This proves that $\mathcal{F}(r,n,l)$ and $\mathcal{N}(r,n+l,l)$ are indeed isomorphic as schemes.

\bigskip

The case $r=1$ is especially interesting. Recall that a framed torsion free sheaf of rank 1 is the same as the ideal sheaf of a zero dimensional scheme $Z\subset\mathbb{C}^2\simeq\mathbb{P}^2\setminus\ell_\infty$. Thus a framed flag of sheaves with numerical invariants $(1,n,l)$ can be regarded as a pair $(Z_1,Z_2)$ of zero dimensional schemes of $\mathbb{C}^2$ of length $n$ and $n+l$, respectively, and such that $Z_1\subset Z_2$.

In other words, $\calf(1,n,l)$ coincides with the nested Hilbert scheme ${\rm Hilb}^{n,n+l}(\C^2)$ of points on the plane. When $r=1$, the stability of the ADHM datum $(A,B,I,J)$ forces $J=0$ (see \cite[Proposition 2.8]{nakajima}). Thus the enhanced ADHM equations are reduced to
\begin{align*}
 [A,B] =0,\quad [A',B'] =0,\quad AF-FA'=0,\quad BF-FB'=0,
\end{align*} 
and provide the ADHM construction of the nested Hilbert scheme ${\rm Hilb}^{n,n+l}(\C^2)$.

\begin{obs}
The ADHM construction of the nested punctual Hilbert scheme, in which case the endomorphisms $A$, $B$, $A'$, and $B'$ are nilpotent, is also discussed by Bulois and Evain, cf. \cite[Section 3]{BE}.
\end{obs}

Another immediate consequence of Theorem \ref{teo:N-eh-suave}, Lemma \ref{connect}, and Theorem \ref{thm-isom}, we obtain the main claim of this paper, namely:

\begin{cor} \label{cor}
$\calf(r,n,1)$ is an irreducible, nonsingular quasi-projective variety of dimension $2rn+r+1$.
\end{cor}

Next, consider the morphism of schemes
$$ \mathfrak{s} ~ : ~ \mathcal{F}(r,n,l) \to \mathcal{M}(r,n)\times S^l(\mathbb{C}^2) $$
given by 
$$ \mathfrak{s}(E,F,\varphi) := \left( (F,\varphi) , [F/E] \right), $$
where $S^l(\mathbb{C}^2)$ denotes $l$-th symmetric power of $\mathbb{C}^2$, and is regarded as the moduli space of 0-dimensional sheaves on $\mathbb{C}^2$; $[F/E]$ denotes the S-equivalence class of the quotient sheaf $F/E$.

Let $\mathcal{M}^{0}(r,n)$ be the open subset of $\mathcal{M}(r,n)$ consisting of framed locally free sheaves; this is nonempty if and only if $r\ge2$; let also $S^l(\mathbb{C}^2)^0$ be the open subset of $S^l(\mathbb{C}^2)$ consisting of $l$-tuples of distinct points in $\mathbb{C}^2$. Finally, let 
$$ \mathcal{F}^{00}(r,n,l) := \left\{ (E,F,\varphi) ~|~ 
F \mbox{ is locally free, and } 
{\rm supp}(F/E) \mbox{ has } l \mbox{ distinct points} ~ \right\}. $$

When $r\ge2$, the restriction of $\mathfrak{s}$ to $\mathcal{F}^{00}(r,n,l)$ is surjective onto $\mathcal{M}^{0}(r,n)\times S^l(\mathbb{C}^2)^0$. Indeed, given a pair $\left( (F,\varphi) , (x_1,\dots,x_l) \right)$ with $x_{k}$ distinct, it is easy to see that there are epimorphisms of sheaves
$$ F \twoheadrightarrow \oplus \bigoplus_{k=1}^{l} \mathcal{O}_{x_k}. $$
Moreover, note that the fibres of $\mathfrak{s}$ are given by
$$ \Hom\left(F,\bigoplus_{k=1}^{l} \mathcal{O}_{x_k}\right) \simeq \mathbb{P}(F_{x_1})\times\cdots\times\mathbb{P}(F_{x_l}). $$
It follows that $\mathcal{F}^{00}(r,n,l)$ is an irreducible, nonsingular open subset of $\mathcal{F}(r,n,l)$, and
\begin{align*}
\dim \mathcal{F}^{0}(r,n,l) = & 
\dim \mathcal{M}^{0}(r,n) + \dim S^l(\mathbb{C}^2)^0 + l(r-1) \\
= & 2rn + 2l + l(r-1) = 2rn + l(r+1).
\end{align*}

Now let
$$ \mathcal{F}^{0}(r,n,l) := \left\{ (E,F,\varphi) ~|~ F \mbox{ is locally free } \right\}, $$
and assume, in addition, that $l\le r$. In this case, the morphism $\mathfrak{s}$ provides, by the same argument as in the previous paragraph, a surjective morphism 
$$ \mathcal{F}^{0}(r,n,l) \to \mathcal{M}^{0}(r,n)\times S^l(\mathbb{C}^2). $$
This shows that $\mathcal{F}(r,n,l)$ is singular when $r\ge2$ and $l\le r$.  

Finally, Theorem \ref{thm-isom} also indicates the possibility of comparing the morphism obtained in Lemma \ref{lem2} with the one described in Lemma \ref{lem:q-sobrejetiva}. Indeed, there is a commutative diagram of schemes
$$ \xymatrix{
\mathcal{N}(r,n+l,l) \ar[rr]^{\mathfrak{b}}_{\sim} \ar[dr]^{\mathfrak{q}} & & \mathcal{F}(r,n,l) \ar[ld]^{\mathfrak{p}} \\
& \mathcal{M}(r,n) &
} $$
It follows that the fibres of $\mathfrak{q}$ and $\mathfrak{p}$ are isomorphic, hence the equations on the right hand side of (\ref{eq:lem-enhADHMeqobt02}) provide an ADHM-type description of the \emph{framed quot scheme}:
$$ {\rm Quot}^l_{\infty}(F) := \{ (Q,q)\in {\rm Quot}^l(F) ~|~ {\rm supp}(Q)\cap\ell_\infty=\emptyset \} .$$


\section{Geometric structures on $\mathcal{N}(r,c,c')$} \label{geometry}

The goal of this sections is to study geometric structures on the moduli space of framed flags of sheaves, motivated by the fact that the moduli space of framed torsion free sheaves on $\p2$ is known to be a hyperk\"ahler manifold.

Recall that a \emph{hyperk\"ahler manifold} is a Riemannian manifold $(M,g)$ equipped with three parallel complex structures $(\Gamma_1,\Gamma_2,\Gamma_3)$ satisfying the usual quaternionic relations; in addition, each 2-form $\omega_k(\cdot,\cdot):=g(\Gamma_k\cdot,\cdot)$ is a K\"ahler form for the K\"ahler manifold $(M,g,\Gamma_k)$. One can then define a symplectic form $\Omega:=\omega_2+i\omega_3$, which is holomorphic with respect to the complex structure $\Gamma_1$; the triple $(M,\Gamma_1,\Omega)$ is called the \emph{holomorphic symplectic manifold} associated with the hyperk\"ahler ma\-ni\-fold $(M,g,\Gamma_1,\Gamma_2,\Gamma_3)$.

\begin{dfn}
A \emph{pre-hyperk\"ahler manifold} is a K\"ahler ma\-ni\-fold $(M,g,\Gamma)$ equipped with a pair of closed 2-forms $(\omega_1,\omega_2)$ satisfying
\begin{equation}\label{pre-hk}
\omega_2(\cdot,\cdot) = \omega_3(\cdot,\Gamma\cdot) .
\end{equation}
\end{dfn}

Given a pre-hyperk\"ahler manifold $(M,g,\Gamma,\omega_2,\omega_3)$, one can define the closed 2-form $\Omega:=\omega_2+i\omega_3$; condition (\ref{pre-hk}) implies that
$$ \Omega(\cdot,\Gamma\cdot) = i\Omega(\cdot,\cdot) $$
hence $\Omega$ is holomorphic with respect to $\Gamma$. This observation motivates the following definition.

\begin{dfn}
A \emph{holomorphic pre-symplectic manifold} is a triple $(M,\Gamma,\Omega)$ consisting a complex manifold $(M,\Gamma)$ equipped with a holomorphic pre-symplectic structure $\Omega$.
\end{dfn}

The holomorphic pre-symplectic manifold $(M,\Gamma,\Omega)$ decribed in the paragraph before the previous definition is called the holomorphic pre-symplectic manifold associated with the pre-hyperk\"ahler manifold $(M,g,\Gamma,\omega_2,\omega_3)$. Note that $\Omega$ is non-degenerate if and only if both $\omega_2$ and $\omega_3$ are non-degenerate.

We will prove below that $\calf(r,n,1)$ admits the structure of a pre-hyperk\"ahler manifold; this is done by embedding it into a hyperk\"ahler manifold.


\subsection{The hyperk\"ahler manifold $\mathcal W$}

\indent Consider the following vector space
$$
\mathbb{X} = End(V)^{\oplus 2}\oplus Hom(W,V)\oplus Hom(V,W) \oplus End(V')^{\oplus 2} \oplus Hom(V',V) \oplus Hom(V,V').
$$

Define in $\mathbb{X}$ the following equations
\begin{equation}
\label{eq:w}
[A,B] + IJ +FG= 0, \quad [A',B'] -GF = 0,
\end{equation}
where $X=(A,B,I,J,A',B',F,G)\in\mathbb{X}$. A vector $X\in\mathbb{X}$ is called \textit{stable} if it satisfies the conditions $(S.1)$ and $(S.2)$ of Lemma \ref{lem:thetastab-equiv-stab}, i.e., if $F$ is injective and the ADHM data given by $(A,B,I,J)$ is stable. Let
\begin{equation*}
\mathbb{W}=\mathbb{W}(r,c,c'):= \{X\in\mathbb{X}: X \mbox{ satisfies \eqref{eq:w} and X is stable}\}.
\end{equation*}

Note that the $\mathcal{G}$-action in \eqref{eq:acaolivre} is free on $\mathbb{W}$ and that the equations in \eqref{eq:w} are preserved by the $\mathcal{G}$-action in \eqref{eq:acaolivre}. Indeed the freeness of this action has been already proved in Proposition \ref{prop:acaolivre}. Analogously to Lemma \ref{lem:preservedbyG}, one can prove that the equations in \eqref{eq:w} are preserved by the $\mathcal{G}$-action \eqref{eq:acaolivre}. The same is true if the one consider the action of $\mathcal{U}:=U(V)\times U(V')$ on $\mathbb{W}(r,c,c')$ given by \eqref{eq:acaolivre}, i.e.,
\begin{equation}
\begin{array}{ccl}
\label{eq:w-acaolivre}
\mathcal{U}\times \mathbb{X} & \longrightarrow & \mathbb{W}\\
(h,h',X)                     & \longmapsto     & (hAh^{-1}, hBh^{-1}, hI, Jh^{-1}, h'A'h'^{-1}, h'B'h'^{-1}, hFh'^{-1},h'Gh^{-1})
\end{array}
\end{equation}
is a free action of $\mathcal{U}$ on $\mathbb{W}$ which preserves the equations \eqref{eq:w}. The moduli space of stable points of $\mathbb{W}$, $\mathcal{W}(r,c,c')$, can be constructed using Geometric Invariant Theory techniques. Moreover, the moduli space of framed stable representations of the enhanced ADHM quiver, $\mathcal{N}(r,c,c')$ is embedded in $\mathcal{W}(r,c,c')$. In fact, if $X\in\mathbb{X}$ is stable, then $X$ satisfies
\begin{align}
[A,B]+IJ =0,\quad [A',B']=0,\quad G\equiv 0. \nonumber
\end{align}
Thus, $X$ satisfies the equations \eqref{eq:w} and $X\in\mathbb{W}$. In this section, it will be proved that the moduli space $\mathcal{W}(r,c,c')$ can be obtained by a hyperkähler reduction. In other words, $\mathcal{N}(r,c,c')$ is embedded in the hyperkähler variety $\mathcal{W}(r,c,c')$. The hyperkähler reduction is presented in details below.\\
\indent In order to prove that $\mathcal{W}(r,c,c')$ is a hyperkähler variety, consider the following equations
\begin{equation}
\label{eq:U-w}
\left\{\begin{array}{lcl}
[A,A^{\dagger}] + [B,B^{\dagger}] + II^{\dagger} - J^{\dagger}J+FF^{\dagger} - G^{\dagger}G & = & 0 \\ 
\mbox{}[A^{\prime},A^{\prime\dagger}] + [B',B^{\prime\dagger}] -F^{\dagger}F + GG^{\dagger} & = & 0
\end{array}\right. ,
\end{equation}
where $A^{\dagger}$ denotes the hermitian adjoint of the map $A$. It will be proved that the equations \eqref{eq:w} and \eqref{eq:U-w} can be obtained as a hyperkähler moment map $\mu$. Thus, one can view the hyperkähler variety $\widetilde{\mathcal{W}}:=\mu^{-1}(0)/\mathcal{U}$. However, it follows from the Kempf--Ness Theorem that the moduli space $\mathcal{W}(r,c,c')$ obtained above is isomorphic to the hyperkähler variety $\widetilde{\mathcal{W}}$. \\
\indent Indeed, define on $\mathbb{W}$ the hermitian metric
\begin{equation*}
\langle\mbox{ , }\rangle: T\mathbb{W}\times T\mathbb{W}  \longrightarrow  \mathbb{C}
\end{equation*}
given by
\begin{align*}
\langle x_1,x_2\rangle  = \frac{1}{2}tr(& a_1a^{\dagger}_{2} + a_2a^{\dagger}_{1} + b_1b^{\dagger}_{2} + b_2b^{\dagger}_{1} +i_1i^{\dagger}_{2} + i_2i^{\dagger}_{1} + j^{\dagger}_{2}j_1 + j^{\dagger}_{1}j_2 + \nonumber\\ 
&  a'_1a^{\prime\dagger}_{2} + a'_2a^{\prime\dagger}_{1} + b'_1b^{\prime\dagger}_{2} + b'_2b^{\prime\dagger}_{1} + f^{\dagger}_{2}f_1 + f^{\dagger}_{1}f_2 + g_1g^{\prime\dagger}_{2} + g_2g^{\prime\dagger}_{1}),
\end{align*}
where $x_1 = (a_1,b_1,i_1,j_1,a'_1,b'_1,f_1,g_1)$ and $x_2 = (a_2,b_2,i_2,j_2,a'_2,b'_2,f_2,g_2)$. Define in $T\mathbb{W}$ the following complex structures. Let $x=(a,b,i,j,a',b',f,g)\in T\mathbb{W}$,
\begin{equation*}
\left\{\begin{array}{lll}
\Gamma_1(x) & = & (\sqrt{-1}a,\sqrt{-1}b,\sqrt{-1}i,\sqrt{-1}j,\sqrt{-1}a',\sqrt{-1}b',\sqrt{-1}f,\sqrt{-1}g)\\
\Gamma_2(x) & = & (-b^{\dagger},a^{\dagger},-j^{\dagger},i^{\dagger},-b^{\prime\dagger},a^{\prime\dagger},-g^{\dagger},f^{\dagger})\\
\Gamma_3(x) & = & \Gamma_1\circ\Gamma_2(x)
\end{array}\right. .
\end{equation*} 

It is not difficult to check that $(\mathbb{W}, \langle\mbox{ , }\rangle, \Gamma_1, \Gamma_2, \Gamma_3)$ is a hyperkähler manifold.

\begin{pps}\label{W}
$\mathcal{W}(r,c,c')$ is a smooth hyperk\"ahler manifold of (complex) dimension $2c(r+c')$.
\end{pps}

\begin{proof}
	It is straightfoward to check the $\mathcal{U}$-action \eqref{eq:w-acaolivre} satisfies
	\begin{align*}
	\langle(h,h')\cdot u, (h,h')\cdot v\rangle = \langle u,v\rangle
	\end{align*}
	for all $(h,h')\in\mathcal{U}$ and preserves the complex structures $\Gamma_n$, $n\in\{1,2,3\}$, i.e., it satisfies
	\begin{equation*}
	\Gamma_n((h,h')\cdot u) = (h,h')\cdot\Gamma_n(u)
	\end{equation*} 
	for all $n\in\{1,2,3\}$ and for all $(h,h')\in\mathcal{U}$. Let 
	$$
	(\xi,\xi')\in \mathfrak{u}=\mathfrak{u}(V)\times\mathfrak{u}(V'):= \{(\xi,\xi')\in GL(V)\times GL(V'); \xi+\xi^{\dagger}=\xi'+\xi^{\prime\dagger}=0\}.
	$$
	
	One can compute the fundamental vector field $(W_{(\xi,\xi')})$ as following.\\
	\indent Let $W\in\mathbb{W}$ and
	$$
	(W_{(\xi,\xi')})_W=d\Psi_W(1_{V},1_{V'})(\xi,\xi')
	$$
	
	where
	\begin{equation*}
	\begin{array}{llll}
	\Psi_W: & \mathcal{U} & \longrightarrow & \mathbb{W}\\
	&  (h,h')     & \longmapsto     & (h,h')\cdot W
	\end{array}
	\end{equation*}
	
	and consider the smooth curve
	\begin{equation*}
	\gamma:(-\epsilon,\epsilon)\longrightarrow\mathcal{U}
	\end{equation*}
	
	given by the ODE
	\begin{equation*}
	\left\{\begin{array}{lll}
	\gamma(0) & = & (1_{V},1_{V'})\\
	\frac{d}{dt}(\gamma)|_{t=0}&=&(\xi,\xi')
	\end{array}\right. .
	\end{equation*}
	
	Thus, 
%
	
%
%
%
	\begin{align*}
	(W_{(\xi,\xi')})_W =& \frac{d}{dt}(\Psi_W\circ\gamma)|_{t=0} \nonumber \\
	&= ([\xi, a], [\xi, b], \xi i, -j\xi, [\xi, a'], [\xi, b'], \xi f-f\xi', \xi'g-g\xi).
	\end{align*}
	
	Now, in order to construct a moment map, one can prove that the Kähler forms $\omega_n$, $n\in\{1,2,3\}$ are exact, i.e. that there exists a $1$-form $\theta_{n}$ in $\mathbb{W}$ such that $\omega_n = d\theta_n$, for all $n\in\{1,2,3\}$. First, it is not difficult to check that for all $x_1$, $x_2\in T\mathbb{W}$
	\begin{align*}
	\omega_1(x_1,x_2) :=& \langle \Gamma_1 x_1, x_2\rangle \nonumber \\
	&= \frac{\sqrt{-1}}{2}tr(a_1a^{\dagger}_{2} - a_2a^{\dagger}_{1} + b_1b^{\dagger}_{2} - b_2b^{\dagger}_{1} +i_1i^{\dagger}_{2} - i_2i^{\dagger}_{1} + j^{\dagger}_{2}j_1 - j^{\dagger}_{1}j_2 + \nonumber\\ 
	&  a'_1a^{\prime\dagger}_{2} - a'_2a^{\prime\dagger}_{1} + b'_1b^{\prime\dagger}_{2} - b'_2b^{\prime\dagger}_{1} + f^{\dagger}_{2}f_1 - f^{\dagger}_{1}f_2 + g_1g^{\prime\dagger}_{2} - g_2g^{\prime\dagger}_{1})	
	\end{align*}
	
	Let $\pi_1:\mathbb{W}\longrightarrow\mathbb{W}$ given by $\pi_1(a,b,i,j,a',b',f,g)=(a,0,0,0,0,0,0,0)$, one can introduce the following $2$-form
	\begin{align*}
	d\pi_1\wedge d\pi_1^{\dagger}((a_1,b_1,i_1,j_1,a'_1,b'_1,f_1,g_1),(a_2,b_2,i_2,j_2,a'_2,b'_2,f_2,g_2)) =& a_1a_2^{\dagger} - a_2a_1^{\dagger}.
	\end{align*}
	
	Defining $\pi_i:\mathbb{W}\longrightarrow\mathbb{W}$ as the projection in the i-th coordinate, one can write $\omega_1$ as
	\begin{align*}
	\omega_1 =& d\theta_1 
	\end{align*}
	
	Where, $\theta_1$ is given by
	\begin{align*}
	\theta_1 =& \frac{\sqrt{-1}}{2}tr(\pi_1\wedge d\pi_1^{\dagger} +\pi_2\wedge d\pi_2^{\dagger} +\pi_3\wedge d\pi_3^{\dagger} +\pi_4\wedge d\pi_4^{\dagger} \nonumber \\
	& +\pi_5'\wedge d\pi_5^{\prime\dagger} +\pi_6'\wedge d\pi_6^{\prime\dagger} +\pi_7\wedge d\pi_7^{\dagger} +\pi_8\wedge d\pi_8^{\dagger}).
	\end{align*}
	
	
	Identifying $\mathfrak{u}(V)\times\mathfrak{u}(V')\cong\mathfrak{u}(V)^{\ast}\times\mathfrak{u}(V')^{\ast}$ via the inner product $(a,b)=tr(ab^{\dagger})$,  one obtains the moment map
	\begin{align*}
	\mu_1(x) :=& -\theta_1(W_{(\xi,\xi')})_w\\
	=& ((\frac{1}{2\sqrt{-1}}[a,a^{\dagger}] + [b,b^{\dagger}] + ii^{\dagger}-j^{\dagger}j +ff^{\dagger}-g^{\dagger}g)),\nonumber \\(          &\frac{1}{2\sqrt{-1}})([a',a^{\prime\dagger}]+[b',b^{\prime\dagger}]-f^{\dagger}f+gg^{\dagger})).
	\end{align*}
	Repeating this procedure for $\omega_2$ and $\omega_3$, one can find 
	moment maps
	\begin{align*}
	\mu_2(x) =& \frac{-1}{2}(([a,b]+[a^{\dagger},b^{\dagger}+ij-j^{\dagger}i^{\dagger}+fg-g^{\dagger}f^{\dagger}]),\nonumber \\
	& ([a',b']+[a^{\dagger},b^{\dagger}]-gf+f^{\dagger}g^{\dagger}))
	\end{align*}
	and
	
	\begin{align*}
	\mu_3(x) =& \frac{-1}{2\sqrt{-1}}(([a,b]-[a^{\dagger},b^{\dagger}+ij+j^{\dagger}i^{\dagger}+fg+g^{\dagger}f^{\dagger}]),\nonumber \\
	& ([a',b']-[a^{\dagger},b^{\dagger}]-gf-f^{\dagger}g^{\dagger})).
	\end{align*}
	
	Then, defining the moment map
	\begin{align*}
	\mu_{\mathbb{C}}(x) =& (\mu_2+\sqrt{-1}\mu_3)(x)\nonumber \\
	=& -(([a,b]+ij+fg),[a',b']-gf).
	\end{align*} 
	
	it follows from Theorem \cite[Theorem 3.35]{nakajima} that
	$$
	\mathcal{W}=\frac{\mu_1^{-1}(0)\cap\mu_{\mathbb{C}}^{-1}(0)\cap\mathbb{W}}{\mathcal{U}}
	$$
	
	has a hyperkähler structure induced from $(\mathbb{W},\langle , \rangle, \Gamma_1, \Gamma_2, \Gamma_3)$, since the $\mathcal{U}$-action acts freely on the stable points of $\mathbb{W}$. Moreover, its real dimension is given by
	\begin{align*}
	\dim_{\mathbb{R}}\mathcal{W}(r,c,c') = 4(c^2+rc+c^{\prime 2}+cc')- 4 (c^2-c^{\prime2}) = 4(rc+cc')
	\end{align*} This concludes the proof that  $(\mathcal{W}(r,c,c'),\langle\mbox{ , }\rangle, \Gamma_1,\Gamma_2,\Gamma_3)$ is a hyperk\"ahler manifold.
\end{proof}

Note that the moduli space $\mathcal{N}(r,c,c')$ of stable framed representations of the enhanced ADHM quiver is a subvariety of the hyperk\"ahler manifold. Indeed, let $$X=(A,B,I,J,A',B',F,G)$$ be a stable representation of the enhanced ADHM quiver. It follows from the enhanced ADHM equations and from Corollary \ref{corolario-lema222} that $X$ satisfies the equations \eqref{eq:w}. Therefore, $X\in\mathbb{W}$ and $\mathcal{N}(r,c,c')$ is a subvariety of $\mathcal{W}(r,c,c')$, embedded by the natural inclusion.


\subsection{The pre-hyperk\"ahler structure on $\caln(r,c,1)$}

\indent In this section, one can find the consequences of the fact that the moduli space $\mathcal{N}(r,c,1)$ is a subvariety of the hyperkähler manifold $\mathcal{W}(r,c,1)=(\mathcal{W}(r,c,1),\langle\mbox{ , }\rangle, \Gamma_1,\Gamma_2,\Gamma_3)$. It was proved in the last section that this is true for the general case $\mathcal{N}(r,c,c')$. However, here it is fixed the moduli space of framed stable representations of the ADHM quiver of numerical type $(r,c,1)$, because this is the only case in which the variety is smooth. First, note that there exists the inclusion map
\begin{equation*}
\begin{tikzcd}
\mathcal{N}(r,c,1)\arrow[hookrightarrow]{r}{\iota} & (\mathcal{W}(r,c,1),\langle\mbox{ , }\rangle, \Gamma_1,\Gamma_2,\Gamma_3).
\end{tikzcd}\end{equation*}

Hence, associated with this inclusion, there exists a complex structure on $\mathcal{N}(r,c,1)$ inherited by the pull-back, $\iota^*\Gamma_1$, and a closed degenerate $2$-form $\Omega=\iota^*\omega_2+\sqrt{-1}\iota^*\omega_3$. Indeed, let $(a,b,i,j,a',b',f,0)\in\mathcal{N}(r,c,1)$. Thus,
\begin{align*}
\iota^*\Gamma_1(a,b,i,j,a',b',f,0) =& \Gamma_1(\iota_*a,\iota_*b,\iota_*i,\iota_*j,\iota_*a',\iota_*b',\iota_*f,0)\nonumber \\
=& (\sqrt{-1}a,\sqrt{-1}b,\sqrt{-1}i,\sqrt{-1}j,\sqrt{-1}a',\sqrt{-1}b',\sqrt{-1}f,0)
\end{align*}
is clearly a complex structure on $\mathcal{N}$. Moreover, let $x_1=(a_1,b_1,i_1,j_1,a'_1,b_1',f_1,0)$ and $x_2=(a_2,b_2,i_2,j_2,a'_2,b_2',f_2,0)$ in $\mathcal{N}$. It is easy to check that $(\mathcal{N}, \iota^*\langle\mbox{ , }\rangle, \iota^*\Gamma_1 )$ has a Kähler structure. The $2$-form $\Omega$ is given by,
\begin{align*}
\Omega(x_1,x_2) =& (\iota^*\omega_2+\sqrt{-1}\iota^*\omega_3)(x_1,x_2)  \nonumber\\
=& (\omega_2+\sqrt{-1}\omega_3)(\iota_*x_1,\iota_*x_2) \nonumber \\ 
&= tr(-a_2b_1+b_2a_1-i_2j_1+i_1j_2-a'_2b'_1+b'_2a'_1).
\end{align*} 

Note that taking $u=(0,0,0,0,0,0,f,0)\in T\mathcal{N}$, $\Omega_{X}(u,v)\equiv 0$ for all $v\in T\mathcal{N}$, i.e., $\Omega$ is in fact a degenerate $2$-form. Also, it is easy to check that the $2$-forms $\iota^*\omega_2$ and $\iota^*\omega_3$ satisfy
\begin{equation*}
\left\{\begin{array}{lll}
\iota^*\omega_2(u,v) & = & \iota^*\omega_3(u,\Gamma_1v)\\
\iota^*\omega_3(u,v) & = & -\iota^*\omega_2(u,\Gamma_1v)
\end{array}\right. .
\end{equation*}
In other words, $\mathcal{N}(r,c,1)$ admits the structure of a pre-hyperk\"ahler manifold.


\subsection{An explicit example}

We consider now the case $r=1$, $c=2$ to precisely determine the degeneration locus of the 
closed holomorphic $2$-form $\Omega$ defined above, that is for which points $X\in\mathcal{N}(1,2,1)$ the linear map
\begin{equation*} \begin{array}{lcl}
T_X\mathcal{N}(1,2,1) & \longrightarrow & \left( T_X\mathcal{N}(1,2,1) \right)^* \\
             u        & \longmapsto     & \Omega_X(u,\cdot)
\end{array}
\end{equation*}
fails to be an isomorphism.

First, we need to prove the following auxiliary Lemma.

\begin{lem}\label{lemma-auxiliar}
Let $X=(W,V,V',A,B,I,J,A',B',F,G)$ be a framed stable representation of the enhanced ADHM quiver of numerical type $(1,2,1)$. Thus, there exists a change of basis for $V$ such that
\begin{itemize}
\item[(i)] $A=\begin{bmatrix}A' & 0 \\ 0 & a_2\end{bmatrix}$, $B=\begin{bmatrix}B' & 0 \\ 0 & b_2\end{bmatrix}$, $F=\begin{bmatrix} 1\\ 0\end{bmatrix}$, if $A$ and $B$ are diagonalizable;
\item[(ii)] $A=\begin{bmatrix}A' & 1 \\ 0 & A'\end{bmatrix}$, $B=\begin{bmatrix}B' & B_{12} \\ 0 & B'\end{bmatrix}$, $F=\begin{bmatrix} 1\\ 0\end{bmatrix}$, if $A$ and $B$ are not diagonalizable;
\item[(iii)] $A=A'\begin{bmatrix} 1 & 0 \\ 0 & 1\end{bmatrix}$, $B=\begin{bmatrix}B' & 1 \\ 0 & B'\end{bmatrix}$, $F=\begin{bmatrix} 1\\ 0\end{bmatrix}$, if $A$ is diagonalizable and $B$ is not diagonalizable.
\end{itemize}
\end{lem}

\begin{proof}
 	First, note that
 	\begin{equation*}
 	\left\{\begin{array}{c}
 	AF(1)-F(1)A'=0\\
 	BF(1)-F(1)B'=0
 	\end{array}\right.,
 	\end{equation*}
 	i.e., $F(1)\in V$ is an eigenvector of $A$ and $B$. Also, $A'$ and $B'$ are eigenvalues of $A$ and $B$, respectively, associated with the vector $F(1)$. Suppose that both $A$ and $B$ are diagonalizable. Therefore, there exists a vector $w\in V$ that satisfies
 	$$Aw- a_2\cdot1_{V}w=0,$$
 	where $A'\neq a_2\in V'$. There exists a change of basis for V such that
 	\begin{equation*}
 	F=\begin{bmatrix} 1 \\ 0\end{bmatrix},\quad w=\begin{bmatrix}0 \\ 1\end{bmatrix}.
 	\end{equation*}
 	Thus, we obtain
 	\begin{equation*}
 	0= AF-FA' = 
 	\begin{bmatrix} 
 	a_{11} & a_{12}\\ 
 	a_{21} & a_{22} 
 	\end{bmatrix}
 	\begin{bmatrix}
 	1\\
 	0
 	\end{bmatrix} - 
 	\begin{bmatrix} 
 	1\\ 
 	0
 	\end{bmatrix}A' = 
 	\begin{bmatrix}
 	a_{11}-A'\\ 
 	a_{21}
 	\end{bmatrix}.
 	\end{equation*}
 	Hence
 	\begin{equation*}
 	A=\begin{bmatrix}
 	A' & a_{12}\\
 	0  & a_{22}
 	\end{bmatrix}.
 	\end{equation*}
 	Analogously, we obtain
 	\begin{equation*}
 	B=\begin{bmatrix}
 	B' & b_{12}\\
 	0  & b_{22}
 	\end{bmatrix}.
 	\end{equation*}
 	
 	Since $a_2$ is an eigenvalue of A associated with the vector $w\in V$, we get
 	\begin{equation*}
 	\begin{bmatrix} 
 	a_{12}\\ 
 	a_{22}-a_2 
 	\end{bmatrix}=\begin{bmatrix}
 	0\\0
 	\end{bmatrix}
 	\end{equation*}
 	Therefore, 
 	\begin{equation*}
 	A=\begin{bmatrix}
 	A' & 0 
 	\\ 0 & a_2
 	\end{bmatrix}.
 	\end{equation*}
 	Since $r=1$, we have $J=0$. Hence
 	\begin{equation*}
 	\begin{tikzcd}
 	0=[A,B] = 
 	\begin{bmatrix}
 	0 & b_{12}(A'-a_{2})\\ 
 	0 & 0
 	\end{bmatrix}
 	\end{tikzcd}.
 	\end{equation*}
 	Thus $b_{12}=0$ and 
 	\begin{equation*}
 	B=	\begin{bmatrix}
 	B' & 0 \\ 
 	0 & b_{22}
 	\end{bmatrix}.
 	\end{equation*}
 	
 	\indent If both $A$ and $B$ are not diagonalizable, there exists $w\in V$ such that $v:= Aw-A'\cdot 1_{V}$ is an eigenvector of $A$. There exist a change of basis for $V$ such that
 	\begin{align*}
 	F=\begin{bmatrix}1\\0\end{bmatrix},\quad w=\begin{bmatrix}0\\1\end{bmatrix},
 	\end{align*} 
 	and
 	\begin{equation*}
 	v =	\begin{bmatrix}
 	A' & a_{12}\\ 
 	0 & a_{22}
 	\end{bmatrix}
 	\begin{bmatrix}
 	0\\
 	1
 	\end{bmatrix} - 
 	\begin{bmatrix}
 	0\\ 
 	A'
 	\end{bmatrix} = 
 	\begin{bmatrix}
 	a_{12}\\ 
 	a_{22}-A'
 	\end{bmatrix}
 	\end{equation*}
 	Hence,
 	\begin{equation*}
 	\begin{tikzcd}
 	0 = Av-A'\cdot 1_{V}v 
 	= \begin{bmatrix} 
 	a_{12}(a_{22}-A')\\ 
 	(a_{22}-A')^{2} 
 	\end{bmatrix}
 	\end{tikzcd},
 	\end{equation*}
 	i.e., $a_{12}\neq 0$ and $a_{22}=A'$. Thus, 
 	$$
 	A=
 	\begin{bmatrix}
 	A' & a_{12}\\ 
 	0 & A'
 	\end{bmatrix}.
 	$$
 	In the other hand,
 	\begin{equation*}
 	\begin{tikzcd}
 	0 = [A,B] = 
 	\begin{bmatrix} 
 	0 & b_{22}-B'\\ 
 	0 & 0
 	\end{bmatrix}
 	\end{tikzcd}
 	\end{equation*}
 	Therefore 
 	$$
 	B=	\begin{bmatrix}
 	B' & b_{12}\\ 
 	0 & B'
 	\end{bmatrix}. 
 	$$
 	Let 
 	
 	$$
 	S=	\begin{bmatrix}
 	1 & 0 \\ 
 	0 & \frac{1}{a_{12}}
 	\end{bmatrix}.
 	$$ 
 	Then we obtain
%
 	\begin{equation*}
 	SAS^{-1} = 
 	\begin{bmatrix}
 	A' & 1 \\ 
 	0 & A'
 	\end{bmatrix}, \quad
 	SBS^{-1} = 	
 	\begin{bmatrix}
 	B' & b_{12}a_{12} \\ 
 	0  & B'
 	\end{bmatrix}, \quad
 	SF = 
 	\begin{bmatrix}
 	1 \\ 
 	0
 	\end{bmatrix}.
 	\end{equation*}
 	Denoting $B_{12}=b_{12}a_{12}$, this concludes the proof of the case where $A$ and $B$ are both non-diagonalizable. The last case, $A$ diagonalizable and $B$ not diagonalizable, is entirely analogous. 
 \end{proof}
 
 We are finally in position to prove the main result of this section.

\begin{pps}\label{prop28}
Let $\mathcal{N}(1,2,1)$ be the moduli space of framed stable representations of the enhanced ADHM quiver of numerical type $(1,2,1)$. Fix a framed stable representation $X=(A,B,I,J,A',B',F)$. Then the $2$-form $\Omega_X$ defined on $T_X\mathcal{N}(1,2,1)$ is non-degenerate if and only if the matrices associated with the endomorphisms $A$ and $B$ are diagonalizable. 
\end{pps}

\begin{proof}
Recall that if $r=1$, then the map $J\in Hom(V,W)$ must vanish, since $X$ is stable (see \cite[Proposition 2.8]{nakajima}), and recall that if $c'=1$, thus $[A',B']=0$, for all $A',$ $B'\in V'$. Thus, the enhanced ADHM equations are reduced to
\begin{align*}
 [A,B] =0,\quad AF-FA'=0,\quad BF-FB'=0.
\end{align*} 
 	
 	Suppose that $A$ and $B$ are not diagonalizable. Thus, it follows from Lemma \ref{lemma-auxiliar} (ii) that there exists change of basis for $V$ such that
 	%
 	\begin{equation}
 	\label{eq:X-n-diag-01}
 	A=	\begin{bmatrix}
 	A' & 1 \\ 
 	0 & A'
 	\end{bmatrix},\quad 
 	B=	\begin{bmatrix}
 	B' & B_{12} \\ 
 	0 & B'
 	\end{bmatrix},\quad 
 	F=	\begin{bmatrix} 
 	1\\
 	0
 	\end{bmatrix}
 	\end{equation} 
 	In order to $X=(A,B,I,J,A',B',F)$ is a stable representation of the enhanced ADHM quiver, 
 	\begin{equation}
 	\label{eq:X-n-diag-03}
 	I =\left[\begin{array}{c}
 	\mu\\
 	1
 	\end{array}\right].
 	\end{equation}
 	Indeed, $F$ is clearly injective. Furthermore, Consider 
 	\begin{equation*}
 	I=\left[\begin{array}{c}
 	i_1\\
 	i_2\end{array}\right].
 	\end{equation*}
 	
 	Note that if $i_2\neq 0$, then there exists a basis of $W$ such that $i_2=1$ and   
 	\begin{equation}
 	\label{eq:I-n-diag}
 	I=\left[\begin{array}{c}
 	\mu\\
 	1\end{array}\right]
 	\end{equation}
 	for some $\mu\in \mathbb{C}$
 	Indeed, suppose that there exists $0\subset S\subset V$ such that 
 	$$
 	A(S), B(S), I(W)\subset S.
 	$$
 	Let $0\neq v \in I(W)$, thus, there exists $w\in W$ such that 
 	\begin{equation*}
 	v=Iw=\left[\begin{array}{c}
 	i_1\\
 	i_2\end{array}\right]\cdot w = \left[\begin{array}{c}
 	i_1w\\
 	i_2w\end{array}\right] \in S
 	\end{equation*}
 	Moreover, since $v\in S$, $A(v)$ $B(v)\in S$ by hypothesis. Hence
 	\begin{equation*}
 	A(v)= \left[\begin{array}{cc}
 	A' & 1\\
 	0 & A'\end{array}\right]\cdot\left[\begin{array}{c}
 	i_1w\\
 	i_2w\end{array}\right] = \left[\begin{array}{c}
 	A'i_1 +i_2 \\
 	A'i_2\end{array}\right]\cdot w
 	\end{equation*}
 	
 	Analogously, 
 	\begin{equation*}
 	B(v)= \left[\begin{array}{c}
 	B'i_1 +B_{12}i_2 \\
 	B'i_2\end{array}\right]\cdot w
 	\end{equation*}
 	
 	Then, if $i_2=0$, $S=\langle i_1\rangle\subsetneq V$ is a subset such that $A(S), B(S), I(w)\subset S$. Indeed, 
 	\begin{equation*}
 	I(w) = \left[\begin{array}{c}
 	i_1w\\
 	0\end{array}\right] \in S
 	\end{equation*}
 	for all $w\in W$, which means $I(W)\subset S$. Let $s\in S$ given by $s=\lambda i_1$. Thus,
 	\begin{equation*}
 	A(s)= \left[\begin{array}{cc}
 	A' & 1\\
 	0 & A'\end{array}\right]\cdot\left[\begin{array}{c}
 	\lambda i_1\\
 	0\end{array}\right]= \left[\begin{array}{c}
 	A' \lambda i_1\\
 	0
 	\end{array}\right]\in S
 	\end{equation*}
 	
 	for all $s\in S$. And therefore, $A(S)\subset S$. Analogously, $B(S)\subset S$. Thus, in order to prove that $X$ is stable, $i_2\neq 0$ and then there exists a basis of $W$ such that $I$ is given by \eqref{eq:I-n-diag}. This concludes that if $X$ is a framed stable representation of the enhanced ADHM quiver, then $A$, $B$, $I$, $A'$, $B'$, $F$ are of the form \eqref{eq:X-n-diag-01}, \eqref{eq:X-n-diag-03}.\\
 	\indent Now, consider $v\in T_X\mathcal{N}$ given by $v=(a,b,i,j,a',b',f)$ such that $X$ satisfies \eqref{eq:X-n-diag-01} and \eqref{eq:X-n-diag-03}. Then, if follows from Theorem \ref{teo:N-eh-suave} and from the fact that $J=0$ that
 	\begin{align*}
 	j=0,\quad [a,B] + [A,b]=0,\quad fA'+Fa'-aF-Af=0,\quad fB'+Fb'-bF-Bf=0.
 	\end{align*}
 	Then, denoting 
 	\begin{equation*}
 	a=\left[\begin{array}{cc}
 	a_{11} & a_{12}\\
 	a_{21} & a_{22}\end{array}\right],\quad b=\left[\begin{array}{cc}
 	b_{11} & b_{12}\\
 	b_{21} & b_{22}\end{array}\right], \quad i=\left[\begin{array}{c}
 	i_{1} \\
 	i_{1}\end{array}\right], \quad f=\left[\begin{array}{c}
 	f_{1} \\
 	f_{2}\end{array}\right]
 	\end{equation*}
 	one gets 
 	\begin{equation*}
 	\begin{array}{lcl}
 	fA'+Fa'-aF-Af & = & 
 	 \left[\begin{array}{c}
 	a'-a_{11}-f_{2} \\
 	a_{21}\end{array}\right]
 	\end{array} .
 	\end{equation*}
 	Therefore, 
 	\begin{equation}
 	\label{eq:a-b-n-diag-aux01}
 	fA'+Fa'-aF-Af = 0 \Leftrightarrow \left\{\begin{array}{lcl}
 	a'&= & a_{11} + f_{2}\\
 	a_{21} & = & 0
 	\end{array}\right. .
 	\end{equation}
 	Analogously,
 	\begin{equation}
 	\label{eq:a-b-n-diag-aux02}
 	fB'+Fb'-bF-Bf = 0 \Leftrightarrow \left\{\begin{array}{lcl}
 	b'&= & b_{11} + B_{12}f_{2}\\
 	b_{21} & = & 0
 	\end{array}\right.
 	\end{equation}
 	
 	It follows from equations \eqref{eq:a-b-n-diag-aux01} and \eqref{eq:a-b-n-diag-aux02} that
 	\begin{equation*}
 	\begin{array}{lcl}
 	[a,B]+[A,b] 
 	&=& \left[\begin{array}{cc}
 	0 & B_{12}(a_{11}-a_{22})+b_{22}-b_{11}\\
 	0 & 0\end{array}\right] \\
 	\end{array}
 	\end{equation*}
 	Therefore, 
 	\begin{equation}
 	\label{eq:a-b-n-diag-aux03}
 	[a,B]+[A,b] = 0 \Leftrightarrow B_{12}(a_{11}-a_{22})+b_{22}-b_{11} = 0.
 	\end{equation}
 	Thus, $v=(a,b,i,j,a',b',f)\in T_X\mathcal{N}$ is such that
 	\begin{align*}
 	a=\left[\begin{array}{cc}
 	a_{11} & a_{12}\\
 	0 & a_{22}\end{array}\right],&& \quad
 	b=\left[\begin{array}{cc}
 	b_{11} & b_{12}\\
 	0 & b_{22}\end{array}\right], && \quad
 	i=\left[\begin{array}{c}
 	i_{1} \\
 	i_{1}\end{array}\right], \\ 
 	f=\left[\begin{array}{c}
 	f_{1} \\
 	f_{2}\end{array}\right], &&\quad 
 	a'=f_2+a_{11}, && \quad
 	b'=B_{12}f_2+b_{11}.
 	\end{align*}
 	and satisfies \eqref{eq:a-b-n-diag-aux03}.\\
 	\indent Thus, for $u=(a,b,i,j,a',b',f)$ and $v=(\widetilde{a}, \widetilde{b}, \widetilde{i}, \widetilde{a'}, \widetilde{b'}, \widetilde{f})\in T_X\mathcal{N}$, 
 	\begin{equation*}
 	\begin{array}{lcl}
 	\Omega_X(u,v) 
 	& =	& ( 2(-\widetilde{a_{11}}b_{11} + \widetilde{b_{11}}a_{11}) - \widetilde{a_{22}}b_{22} + \widetilde{b_{22}}a_{22} +f_2(\widetilde{b_11}-\widetilde{a_11}) + \widetilde{f_2}(a_{11}-b_{11}) ).
 	\end{array}
 	\end{equation*} 
 	
 	Thus, if $u\in T_X\mathcal{N}$ satisfies
 	\begin{equation}
 	\label{eq:x-n-diag-aux-eq}
 	\left\{\begin{array}{lcl}
 	b_{11} & = &B_{12}a_{11}\\
 	f_2  &   = & 0
 	\end{array}\right.,
 	\end{equation}
 	
 	then
 	\begin{equation*}
 	\begin{array}{lcl}
 	\Omega_X(u,v) & = & ( 2(\widetilde{b_{11}}a_{11}-\widetilde{a_{11}}B_{12}a_{11}) - \widetilde{a_{22}}b_{22} + \widetilde{b_{22}}a_{22}) 
 	\end{array}
 	\end{equation*}
 	
 	Note that it follows from equation \eqref{eq:a-b-n-diag-aux03} that if $B_{12}a_{11}=b_{11}$, then $B_{12}a_{22}=b_{22}$. Thus one obtains
 	\begin{equation*}
 	\begin{array}{lcl}
 	\Omega_X(u,v) & = & 2(-\widetilde{a_{11}}B_{12}a_{11} + \widetilde{b_{11}}a_{11}) - \widetilde{a_{22}}B_{12}a_{22} + \widetilde{b_{22}}a_{22} \\
 	& = & 2a_{11}(\widetilde{b_{11}} -B_{12}\widetilde{a_{11}}) + a_{22}(\widetilde{b_{22}}-B_{12}\widetilde{a_{22}}) 
 	\end{array}.
 	\end{equation*}
 	
 	Since $v\in T_X\mathcal{N}$, 
 	\begin{align*}
 	B_{12}\widetilde{a_{11}}-B_{12}\widetilde{a_{22}}+\widetilde{b_{22}}-\widetilde{b_{11}}  =  0 \quad \Rightarrow\quad \widetilde{b_{22}} - B_{12}\widetilde{a_{22}} = \widetilde{b_{11}} - B_{12}\widetilde{a_{11}} .
 	\end{align*}
 	
 	Therefore
 	\begin{equation*}
 	\begin{array}{lcl}
 	\Omega_X(u,v) & = & 2a_{11}(\widetilde{b_{11}} -B_{12}\widetilde{a_{11}}) + a_{22}(\widetilde{b_{11}} -B_{12}\widetilde{a_{11}})\\
 	& = & (2a_{11}+a_{22})(\widetilde{b_{11}} -B_{12}\widetilde{a_{11}})
 	\end{array}
 	\end{equation*}
 	
 	Then, if $u\in T_X\mathcal{N}$ satisfies \eqref{eq:x-n-diag-aux-eq} and $a_{22}=-2a_{11}$, 
 	\begin{align*}
 	\Omega_X(u,\cdot) \equiv 0
 	\end{align*}
 	
 	Denote by $[u]$ the equivalence class of $u$. It is easy to check that if $a_{11}\neq 0$, $[u] \neq [0]$. Indeed, it follows from Theorem \ref{teo:N-eh-suave} that $[u]=[0]$ if and only if $u\in Im(d_0)$, where
 	\begin{equation*}
 	\begin{array}{lcll}
 	do:& End(V)\oplus End(V') & \longrightarrow & \mathbb{X}\\
 	& (h,h')               & \longmapsto     & ([h,A], [h,B], hI, -Jh, [h',A'], [h',B'], hF-Fh')
 	\end{array}.
 	\end{equation*}
 	
 	Since $c'=1$, $[h',A']=0$. Then, if $a_{11}\neq 0$, since $f_{2}=0$ and $a'=a_{11}+f_{2}$, $a'\neq 0$. Therefore, there is no $(h,h')\in End(V)\oplus End(V')$ such that $d_0(h,h')=u$. In other words, 
 	\begin{align*}
 	[u] = \left[\left(\left[\begin{array}{cc}
 	a_{11} & a_{12}\\
 	0      & -2a_{11}
 	\end{array}\right], 
 	\left[\begin{array}{cc}
 	B_{12}a_{11} & b_{12}\\
 	0      & -2B_{12}a_{11}
 	\end{array}\right],
 	\left[\begin{array}{c}
 	i_1\\
 	i_2
 	\end{array}\right],
 	a_{11}, a_{11}, \left[\begin{array}{c}
 	f_{1}\\
 	0
 	\end{array}\right]\right)
 	\right]\neq [0]
 	\end{align*}
 	is a non-null vector in $T_X\mathcal{N}(1,2,1)$ such that $\Omega_X(u,v)=0$, for all $v\in T_X\mathcal{N}(1,2,1)$. This concludes that if $X$ satisfies \eqref{eq:X-n-diag-01}, and \eqref{eq:X-n-diag-03}, then $\Omega_X$ is degenerate.\\
 	\indent Now suppose that $A$ is dianagolizable and $B$ is non-diagonalizable. Analogously to the previous case, one can check that $u\in T_{X}\mathcal{N}(1,2,1)$ is given by
 	\begin{align*}
 	u=\left(\begin{pmatrix}0 & a_{12}\\ 0 & a_{22}\end{pmatrix},\begin{pmatrix}b_{11} & b_{12}\\ 0 & b_{22}\end{pmatrix}, \begin{pmatrix}i_1\\ i_2\end{pmatrix}, 0, 0, b_{11}+f_2, \begin{pmatrix}f_1\\ f_2\end{pmatrix} \right)
 	\end{align*} and \begin{align*}
 	\Omega_X(u,v)= \widetilde{b_{22}}a_{22} - \widetilde{a_{22}}b_{22} 
 	\end{align*}
 	Hence, taking $a_{22}=b_{22}=0$ e $b_{11}\neq f_2$, i.e., 
 	\begin{align*}
 	u=\left(\begin{pmatrix}0 & a_{12}\\ 0 & 0\end{pmatrix},\begin{pmatrix}b_{11} & b_{12}\\ 0 & 0\end{pmatrix}, \begin{pmatrix}i_1\\ i_2\end{pmatrix}, 0, 0, b_{11}+f_2, \begin{pmatrix}f_1\\ f_2\end{pmatrix} \right)\\
 	\end{align*}
 	one gets $\Omega_x(u,v)=0$ for all $v\in T_X\mathcal{N}(1,2,1)$ with $[u]\neq[0]$. Thus $\Omega_X$ is degenerate.\\
 	\indent In order to conclude the proof, we have now to prove that if $$X=(A,B,I,J,A',B',F)$$ is such that $A$ and $B$ are diagonalizable matrices, then it follows from Lemma \ref{lemma-auxiliar} that there exists a change of basis for $V$ such that
 	\begin{equation}
 	\label{eq:X-diag}
 	A= \left[\begin{array}{cc}
 	A' & 0 \\
 	0   & a_2
 	\end{array}
 	\right],\quad
 	B= \left[\begin{array}{cc}
 	B' & 0 \\
 	0   & b_2
 	\end{array}
 	\right],\quad
 	F =\left[\begin{array}{c}
 	1\\
 	0
 	\end{array}\right].
 	\end{equation}

 	Moreover, analogously to the previous case, one can check that, in order for $X$ to be stable, 
 	\begin{equation*}
 	I =\left[\begin{array}{c}
 	\lambda\\
 	1
 	\end{array}\right], \quad a_1\neq a_2,\quad b_1\neq b_2 
 	\end{equation*}
 	
 	and a vector $u=(a,b,i,j,a',b',f)\in T_X\mathcal{N}$ is given by
 	\begin{equation*}
 	u = \left(\left[\begin{array}{cc}
 	a_{11}              & a_{12}\\
 	f_{2}(a_1-a_2)      & a_{22}
 	\end{array}\right], 
 	\left[\begin{array}{cc}
 	b_{11}                     & \delta a_{12}\\
 	\delta f_{2}(a_1-a_2)      & b_{22}
 	\end{array}\right],
 	\left[\begin{array}{c}
 	i_1\\
 	i_2
 	\end{array}\right],
 	a_{11}, b_{11}, \left[\begin{array}{c}
 	f_{1}\\
 	f_{2}
 	\end{array}\right]\right)
 	\end{equation*}
 	Therefore, one can check that given $u=(a,b,i,j,a',b',f)$, $v=(\widetilde{a},\widetilde{b},\widetilde{i},\widetilde{j},\widetilde{a'},\widetilde{b'},\widetilde{f} )\in T_X\mathcal{N}$,
 	\begin{align*}
 	\Omega_X(u,v) = 2(a_{11}\widetilde{b_{11}} - \widetilde{a_{11}}b_{11})+ a_{22}\widetilde{b_{22}} - \widetilde{a_{22}}b_{22}.
 	\end{align*}
 	
 	Moreover, $\Omega_X(u,\cdot)\equiv 0$ if and only if $a_{11}=b_{11}=a_{22}=b_{22}=0$. Indeed, if $a_{11}=b_{11}=a_{22}=b_{22}=0$, it is trivial that $\Omega_X(u,\cdot)\equiv 0$. Now suppose that $\Omega_X(u,\cdot)=0$ for all $v\in T_X\mathcal{N}$. In particular, by taking
 	\begin{align*}
 	v = \left(\left[\begin{array}{cc}
 	0                   & a_{12}\\
 	f_{2}(a_1-a_2)      & 0
 	\end{array}\right], 
 	\left[\begin{array}{cc}
 	\overline{a_{11}}                     & \delta a_{12}\\
 	\delta f_{2}(a_1-a_2)                 & \overline{a_{22}}
 	\end{array}\right],
 	\left[\begin{array}{c}
 	i_1\\
 	i_2
 	\end{array}\right],
 	0, \overline{a_{11}}, \left[\begin{array}{c}
 	f_{1}\\
 	f_{2}
 	\end{array}\right]\right),\\
 	w = \left(\left[\begin{array}{cc}
 	-\overline{b_{11}}                   & a_{12}\\
 	f_{2}(a_1-a_2)      & -\overline{b_{22}}
 	\end{array}\right], 
 	\left[\begin{array}{cc}
 	0                     & \delta a_{12}\\
 	\delta f_{2}(a_1-a_2)                 & 0
 	\end{array}\right],
 	\left[\begin{array}{c}
 	i_1\\
 	i_2
 	\end{array}\right],
 	-\overline{b_{11}}, 0, \left[\begin{array}{c}
 	f_{1}\\
 	f_{2}
 	\end{array}\right]\right).
 	\end{align*}
 	
 	$\Omega_X(u,v)=\Omega_X(u,w)=0$, but
 	\begin{align*}
 	\Omega_X(u,v) = 2a_{11}\overline{a_{11}} + a_{22}\overline{a_{22}},
 	\end{align*}
 	and this vanishes if and only if $a_{11}=a_{22}=0$. Moreover,
 	\begin{align*}
 	\Omega_X(u,w) = 2b_{11}\overline{b_{11}} + b_{22}\overline{b_{22}}
 	\end{align*}
 	and this vanishes if and only if $b_{11}=b_{22}=0$. Therefore, $\Omega_x(u,\cdot)\equiv 0$ if and only if 
 	\begin{equation*}
 	u = \left(\left[\begin{array}{cc}
 	0              & a_{12}\\
 	f_{2}(a_1-a_2)      & 0
 	\end{array}\right], 
 	\left[\begin{array}{cc}
 	0                     & \delta a_{12}\\
 	\delta f_{2}(a_1-a_2)      & 0
 	\end{array}\right],
 	\left[\begin{array}{c}
 	i_1\\
 	i_2
 	\end{array}\right],
 	0, 0, \left[\begin{array}{c}
 	f_{1}\\
 	f_{2}
 	\end{array}\right]\right).
 	\end{equation*}
 	
 	However, $u=d_0(h,h')$, for
 	\begin{align*}
 	h=&\left[\begin{array}{cc}
 	f_{1}-\frac{a_{12}+i_1(i_1-i_2)-\lambda(a_1-a_2)f_{1}}{\lambda(a_1-a_2)} & -\frac{a_{12}}{(a_1-a_2)}\\
 	f_{2} & i_2-\lambda f_2
 	\end{array}\right],\\ h'=& \frac{a_{12}+i_1-\lambda(a_1-a_2)f_{1}}{\lambda(a_1-a_2)}.
 	\end{align*}
 	Indeed, one can check that
 	\begin{equation*}
 	\begin{array}{l}
 	d_0(h,h') =\\ 
 	\left(\begin{array}{l}(a_1-a_2)H,
 	(b_1-b_2)H,
 	\left[\begin{array}{c}
 	\lambda h_{11}+h_{12}\\
 	\lambda h_{21} + h_{22}
 	\end{array}\right],
 	0,
 	0,
 	0,
 	\left[\begin{array}{cc}
 	h_{11}-h'\\
 	h_{21}
 	\end{array}\right]\end{array}\right)
 	\end{array}  
 	\end{equation*}
 	where
 	\begin{align*}
 	H=\left[\begin{array}{cc}
 	0 & -h_{12}\\
 	h_{21} & 0
 	\end{array}\right]
 	\end{align*}
 	and then, $d(h,h')=u$. This means that $[u]=[0]$. In other words, if $X$ satisfies \eqref{eq:X-diag}, $\Omega_X$ is non-degenerate.
 \end{proof}

Proposition \ref{prop28} implies that the pre-symplectic form $\Omega$ is generically non-degenerate on $\mathcal{N}(1,2,1)$. We believe that the same holds for $\mathcal{N}(1,c,1)$ with arbitrary $c$.


\end{document}